\documentclass[11pt]{amsart}
\usepackage{amscd,amssymb,verbatim}

\theoremstyle{plain}
\newtheorem{thm}{Theorem}
\newtheorem{cor}[thm]{Corollary}
\newtheorem{lem}[thm]{Lemma}
\newtheorem{prop}[thm]{Proposition}

\newcommand{\ep}{\varepsilon}
\newcommand{\al}{\alpha}

\newcommand{\vp}{\varphi}

\newcommand{\bs}{\setminus}

\newcommand{\ol}{\overline}

\newcommand{\inte}{\operatorname{int}}

\newcommand{\N}{{\mathbb N}}
\newcommand{\R}{{\mathbb R}}

\newcommand{\cO}{{\mathcal O}}

\newcommand{\cS}{{\mathcal S}}
\newcommand{\cA}{{\mathcal A}}

\newcommand{\U}{{\mathcal U}}

\begin{document}


\title[Compact disjointness preserving operators]{Compact and weakly compact disjointness preserving operators on spaces of differentiable functions}

\begin{abstract}
A pair of functions defined on a set $X$ with values in a vector space $E$ is said to be disjoint if at least one of the functions takes the value $0$ at every point in $X$. An operator acting between vector-valued function spaces is disjointness preserving if it maps disjoint functions to disjoint functions. We characterize compact and weakly compact disjointness preserving operators between spaces of Banach space-valued differentiable functions.
\end{abstract}

\author{Denny H. Leung}
\address{Department of Mathematics, National University of Singapore, Singapore 119076}
\email{matlhh@nus.edu.sg}
\author{Ya-shu Wang}
\address{Department of Mathematics,
National Central University,
Chungli 32054,
Taiwan, R.O.C.}
\email{wangys@mx.math.ncu.edu.tw}
\thanks{{\em 2010 Mathematics Subject Classification.} Primary 46E40, 46E50, 47B33, 47B38.}
\thanks{{\em Key words and phrases.} Disjointness preserving operators, spaces of vector-valued differentiable functions, compact and weakly compact operators.}
\thanks{Research of the first author was partially supported by AcRF project no.\ R-146-000-130-112}
\maketitle

Let $X, Y$ be topological spaces and let $E, F$ be Banach spaces. Suppose that $A(X,E)$ and $A(Y,F)$ are subspaces of the spaces of continuous functions $C(X,E)$ and $C(Y,F)$ respectively. An operator $T: A(X,E) \to A(Y,F)$ is said to be {\em disjointness preserving} or {\em separating} if $T$ maps disjoint functions to disjoint functions. Here disjointness of two functions $f$ and $g$ in, say, $A(X,E)$ simply means that either $f(x) = 0$ or $g(x) = 0$ for all $x \in X$. In case $E = F = \R$, any algebraic homomorphism $T: C(X) \to C(Y)$ is disjointness preserving. The study of such homomorphisms is classical; see, for example \cite{GJ}.
Disjointness preserving operators on spaces of continuous functions, whether scalar- or vector-valued, have also been investigated extensively; a by-no-means exhaustive list of references includes \cite{A, ABN, BNT, HBN}. We also note that a complete characterization of disjointness preserving operators $T: C(X) \to C(Y)$ is obtained in \cite{J}. A similar result for operators $T: C_0(X) \to C_0(Y)$, $X, Y$ locally compact, was obtained in \cite{LinW}, as well as characterizations of compactness and weak compactness of $T$. These results were extended to the vector-valued case in \cite{JL}. In a parallel direction, many authors have studied algebraic homomorphisms between spaces of scalar-valued differentiable functions \cite{AGL, GGJ, GL}. In the vector-valued case, Araujo \cite{A2} gave a characterization of biseparating maps $T: C^p(X,E) \to C^q(Y,F)$, $p, q < \infty$. (A map $T$ is biseparating if it is a bijection and both $T$ and $T^{-1}$ are disjointness preserving.)

In this paper, we undertake a thorough study of compact and weakly compact disjointness preserving operators $T: C^p(X,E) \to C^q(Y,F)$.
Recall that a linear map $T$ between topological vector spaces is said to be {\em compact}, respectively, {\em weakly compact} if $T$ maps some open set onto a relatively compact, respectively, relatively weakly compact set.
Our results can be compared in particular with those of \cite{GGJ, GL}, which concern homomorphisms on spaces of scalar-valued differentiable functions. (Note, however, that in \cite{GGJ}, a (weakly) compact operator is defined to be one that maps {\em bounded} sets onto relatively (weakly) compact sets.)
A major difference between the present situation and the case of homomorphisms treated in \cite{GGJ, GL} is that the operator $T$ is represented as a series rather than a single term (see Theorem \ref{thm10}).  Because of the interference between terms, extraction of information on the various components of the representation (i.e., the terms $\Phi_k$ and $h$ in Theorem \ref{thm10}) becomes a much more delicate business.

We now briefly summarize the contents of the various sections of the paper.
In \S 1, a variation of the Stone-\v{C}{e}ch compactification procedure is introduced to obtain the support map of a disjointness preserving operator. In \S 2, based on a concentration argument (Proposition \ref{prop9}), a general representation theorem for continuous disjointness preserving operators between spaces of vector-valued differentiable functions is obtained (Theorem \ref{thm10}). The next three sections are devoted to a close study of the support map in various situations.  This study, especially
in the cases $p = q = \infty$ and $p = q < \infty$, requires careful analysis of the map $T$ utilizing the representation in Theorem \ref{thm10}. The result of the analysis is that the support map is locally constant unless $p > q$. In \S 6, the compactness or weak compactness of $T$ is taken into account. First, in Proposition \ref{prop25}, we embed $C^q(Y,F)$ linearly homeomorphically into a product of Banach spaces of (vector-valued) continuous functions. Then, using well-known characterizations of compactness and weak compactness in spaces of continuous functions, we reach the goals of characterizing compact and weakly compact disjointness preserving operators $T: C^p(X,E) \to C^q(Y,F)$ in all cases except where $p > q$. In the final section, it is shown that when $p > q$, not only is the support map $h$ not necessarily locally constant, it does not even have to be differentiable at all points. Nevertheless, we show that $h$ is $C^q$ on a dense open set and deduce characterizations of compact and weakly compact disjointness preserving operators.

The first author thanks Wee-Kee Tang for many discussions concerning the results of the paper.

\section{Compactification}

Let $X$ be a Hausdorff topological space and let $E$ be a Banach
space. Suppose that $A(X)$ is a vector subspace of $C(X)$, the space
of all continuous real-valued functions on $X$.  Assume that $A(X)$
separates points from closed sets: If $x \in X$ and $P$ is a closed
subset of $X$ not containing $x$, then there exists $\vp \in A(X)$
so that $\vp(x) = 1$ and $\vp(P) \subseteq \{0\}$.  Let $A(X,E)$ be
a vector subspace of $C(X,E)$, the space of $E$-valued continuous
functions on $X$, that is $A(X)$-multiplicative: $\vp f \in A(X,E)$
if $\vp \in A(X)$ and $f \in A(X,E)$.  The use of $A(X)$-compactification (defined below) in the study of separating maps is inspired by a similar construction in \cite{A}.

\begin{prop}\label{p1}
Define $i: X \to \R^{A(X)}$ by $i(x)(\vp) = \vp(x)$.  Then $i$ is a homeomorphic embedding.
\end{prop}

\begin{proof}
If $x_1 \neq x_2$, there exists $\vp \in A(X)$ such that $\vp(x_1) \neq \vp(x_2)$.  It follows that $i$ is one-to-one.

Since each $\vp \in A(X)$ is continuous, $i$ is also continuous.

Suppose $x_0 \in X$ and $U$ is an open neighborhood of $x_0$ in $X$.
Set $P = X \bs U$. There exists $\vp \in A(X)$ such that $\vp(x_0) =
1$ and $\vp(P) \subseteq \{0\}$. Now $V = \{a \in \R^{A(X)}: a(\vp)
> 0\}$ is an open neighborhood of $i(x_0)$ in $\R^{A(X)}$.  If $x
\in X$ and $i(x) \in V$, then $\vp(x) = i(x)(\vp) > 0$ and hence $x
\notin P$.  Thus $x \in U$.  This establishes the continuity of
$i^{-1}$ on $i(X)$.
\end{proof}

Denote by $\R_\infty$ the $1$-point compactification of $\R$ and
regard $\R^{A(X)}$ naturally as a subspace of $\R_\infty^{A(X)}$. We
call the closure of $i(X)$ in $\R_\infty^{A(X)}$ the
$A(X)$-compactification of $X$ and denote it by $\cA X$.
Note that $\cA X$ is indeed compact Hausdorff and $X$ is homeomorphic with the dense subspace $i(X)$ in $\cA X$.\\

Assume from hereon that $A(X)$ has the following property ($*$):\\

\noindent For any $N \in \N$, any $\vp_1,\dots, \vp_N$ in $A(X)$ and
any $C^\infty$ function $\Phi: \R^N \to \R$ with
$\sup_{a\in \R^N}|\partial^\xi\Phi(a)| < \infty$ for all $\xi$, the function
\[ \Phi(\vp_1,\dots,\vp_N): x \mapsto \Phi(\vp_1(x),\dots,\vp_N(x)) \text{ belongs to }A(X).\]

We omit the simple proof of the next lemma.

\begin{lem}\label{lem2}
If $P'$ and $Q'$ are sets in $\R^N$ such that $d(P',Q') > 0$, where
$d$ is the usual Euclidean distance on $\R^N$, then there exists a
$C^\infty$ function $\Phi: \R^N \to \R$ with $\sup_{a\in
\R^N}|\partial^\xi\Phi(a)| < \infty$ for all $\xi$ such that
$\Phi(P') \subseteq \{0\}$ and $\Phi(Q') \subseteq \{1\}$.
\end{lem}

\begin{prop}\label{prop2}
Assume that $A(X)$ separates points from closed sets and has
property ($*$).  If $P$ and $Q$ are subsets of $X$ with
$\overline{i(P)}^{\cA X} \cap \overline{i(Q)}^{\cA X} = \emptyset$,
then there exists $\vp \in A(X)$ so that $\vp(P) \subseteq \{0\}$
and $\vp(Q) \subseteq \{1\}$.
\end{prop}

\begin{proof}
Let $P$ and $Q$ be as above.  By compactness of
$\overline{i(P)}^{\cA X}$ and $\overline{i(Q)}^{\cA X}$, there are
two finite collections $\U_1$ and $\U_2$ of basic open sets in
$\R_\infty^{A(X)}$ so that $i(P) \subseteq \cup\U_1$, $i(Q) \subseteq \cup
\U_2$ and $\cup\U_1$, $\cup \U_2$ have disjoint closures in
$\R_\infty^{A(X)}$. Each basic open set $U$ in $\R^{A(X)}_\infty$ is of the
form
\[ \{a \in \R_\infty^{A(X)}: a(\vp) \in V_\vp(U) \text{ for all } \vp \in S(U)\},\]
where $S(U)$ is a finite subset of $A(X)$ and $V_\vp(U)$ is open in
$\R_\infty$ for all $\vp \in S(U)$.
If some $V_\vp(U)$ contains $\infty$, replace it with $V_\vp(U)\cap \R$. The modified collections of basic open sets still cover $i(P)$ and $i(Q)$ respectively.
Let $S = \cup\{S(U): U \in \U_1\cup \U_2\}$.
Then $S$ is a finite subset of $A(X)$.  Denote by $P'$ and
$Q'$ the respective projections of $i(P)$ and $i(Q)$ onto $\R^S$. We
claim that $d(P',Q')> 0$. If the claim is false, there are sequences
$(p_n)$ and $(q_n)$ in $P$ and $Q$ respectively so that
$d(\rho(i(p_n)),\rho(i(q_n)))\to 0$, where $\rho$ is the projection
from $\R^{A(X)}$ onto $\R^S$. This implies that  $|\vp(p_{n}) -
\vp(q_{n})| \to 0$ for all $\vp \in S$. It follows that there is an
increasing sequence $(n_k)$ in $\N$ so that for all $\vp \in S$,
$(\vp(p_{n_k}))$ and $(\vp(q_{n_k}))$ both converge to the same
$r_\vp$ in $\R_\infty$.  Define $r_\vp \in \R_\infty$ arbitrarily
for $\vp \notin S$.  Then $r = (r_\vp)$ lies in the closure of both
$\cup \U_1$ and $\cup \U_2$ in $\R_\infty^{A(X)}$, a contradiction.
So the claim is verified.

Let $\Phi: \R^S \to \R$ be a function given by Lemma \ref{lem2}. By
property ($*$), $\Phi((\vp)_{\vp \in S}) \in A(X)$.  If $x \in P$,
$(\vp(x))_{\vp \in S} = \rho(i(x)) \in P'$ and hence
$\Phi((\vp(x)_{\vp\in S}) = 0$.  Similarly, $\Phi((\vp(x)_{\vp\in
S}) = 1$ for all $x \in Q$.
\end{proof}

Let $Y$ be a Hausdorff topological space and let $F$ be a Banach
space.  Suppose that $A(Y,F)$ is a vector subspace of $C(Y,F)$. Two
functions $f_1, f_2 \in A(X,E)$ are said to be {\em disjoint},
denoted $f_1 \perp f_2$, if for all $x\in X$, either $f_1(x) = 0$ or
$f_2(x) = 0$. Disjointness for functions in $A(Y,F)$ is defined
similarly. A map $T: A(X,E) \to A(Y,F)$ is said to be {\em
separating} or {\em disjointness preserving} if $f_1 \perp f_2$ implies $Tf_1 \perp Tf_2$.  In the
sequel, for $f \in A(X,E)$, let $C(f) = \{x\in X: f(x)\neq 0\}$ and
let $\overline{C}(f)$ be the closure of $i(C(f))$ in $\cA X$.

\begin{prop}\label{prop4}
Let $T: A(X,E)\to A(Y,F)$ be a linear separating map.  Then there is
a continuous function $h : \tilde{Y} \to \cA X$, \[\tilde{Y} = \{y \in Y: \text{
there exists } f\in A(X,E) \text{ with } Tf(y) \neq 0\},\] so that
if $f \in A(X,E)$ and $h(y) \notin \overline{C}(f)$, then $Tf(y) =
0$.
\end{prop}

We will call $h$ the {\em support map} of the operator $T$. The proof of the proposition is decomposed into the next three
lemmas.   For any $y \in Y$, Let $\cO_y$ be the family of all open
subsets $U$ of $\cA X$ so that there exists $f \in A(X,E)$ with
$Tf(y)\neq 0$ and $\overline{C}(f) \subseteq U$.

\begin{lem}\label{lem5}
$\cO_y$ is closed under finite intersections.
\end{lem}

\begin{proof}
Suppose that $U_1, U_2 \in \cO_y$ and $f_1, f_2$ are as in the definition.
Choose open sets $V_1, V_2$ in $\cA X$ so that
\[ \overline{C}(f_1) \cap \overline{C}(f_2) \subseteq V_1  \subseteq
\overline{V_1}^{\cA X} \subseteq V_2 \subseteq \overline{V_2}^{\cA X}\subseteq U_1 \cap U_2.\]
Let $P = i^{-1}(V_1)$ and $Q = i^{-1}(\cA X \bs \overline{V_2}^{\cA
X})$. Then $\overline{i(P)}^{\cA X} \subseteq \overline{V_1}^{\cA
X}$ and $\overline{i(Q)}^{\cA X} \subseteq \cA X\bs V_2$ and hence
$\overline{i(P)}^{\cA X} \cap \overline{i(Q)}^{\cA X} = \emptyset$.
By Proposition \ref{prop2}, there exists $\vp \in A(X)$ so that $\vp(P) \subseteq \{0\}$ and
$\vp(Q)\subseteq \{1\}$.  Then $\vp f_j \in A(X,E)$, $j = 1,2$.  If
$x \in X$ and $(\vp f_1)(x) \neq 0$, then $\vp(x) \neq 0$ and
$f_1(x) \neq 0$.  Hence  $i(x) \notin \overline{C}(f_1)\cap
\overline{C}(f_2)$ and $i(x) \in \overline{C}(f_1)$. Thus $f_2(x) = 0$. Therefore,
$\vp f_1 \perp \vp f_2$ and it follows that either $T(\vp f_1)(y) =
0$ or $T(\vp f_2)(y) = 0$.  Without loss of generality, we may
assume that the former holds.  Then $T(f_1 - \vp f_1)(y) = (Tf_1)(y)
\neq 0$.  Also, if $x \in X$ and $(f_1-\vp f_1)(x) \neq 0$, then
$\vp(x) \neq 1$ and hence $i(x) \in  \overline{V_2}^{\cA X}$. Thus
$\overline{C}(f_1-\vp f_1)\subseteq \overline{V_2}^{\cA X} \subseteq
U_1 \cap U_2$.
\end{proof}

It is easy to see that $\tilde{Y}$ is an open subset of $Y$.
Note that $\cO_y \neq \emptyset$ if $y \in \tilde{Y}$.
By compactness of $\cA X$, $\cap\{\overline{U}^{\cA X}: U \in \cO_y\} \neq \emptyset$ if $y \in \tilde{Y}$.

\begin{lem}\label{lem6}
If $y \in \tilde{Y}$, then $\cap\{\overline{U}^{\cA X}: U \in \cO_y\}$ contains exactly one point.
\end{lem}

\begin{proof}
Suppose that there are distinct points $x_1, x_2$ in
$\cap\{\overline{U}^{\cA X}: U \in \cO_y\}$.   Let $U$ and $V$ be
disjoint open neighborhoods of $x_1$ and $x_2$ respectively in $\cA
X$.  Let $U_1$ and $V_1$ be open neighborhoods of $x_1$ and $x_2$
respectively in $\cA X$ so that $\overline{U_1}^{\cA X} \subseteq U$
and $\overline{V_1}^{\cA X} \subseteq V$.  By Proposition \ref{prop2}, there exists $\vp \in
A(X)$ such that $\vp(i^{-1}(U_1)) \subseteq \{0\}$ and
$\vp(i^{-1}(V_1)) \subseteq \{1\}$.  For any $f \in A(X,E)$,
$\overline{C}(\vp f) \subseteq \cA X \bs U_1$ and
$\overline{C}(f-\vp f) \subseteq \cA X \bs V_1$.  In particular,
$x_1 \notin \overline{C}(\vp f)$ and $x_2 \notin \overline{C}(f-\vp
f)$.  Thus there exists an open neighborhood $W$ of
$\overline{C}(\vp f)$ so that $x_1 \notin \overline{W}^{\cA X}$.  It
follows that $W \notin \cO_y$ and hence $T(\vp f)(y) = 0$.
Similarly, $T(f- \vp f)(y) = 0$.  Hence $Tf(y) = T(\vp f)(y) +
T(f-\vp f)(y) = 0$ for all $f \in A(X,E)$, contrary to the fact that
$y \in \tilde{Y}$.
\end{proof}

For each $y \in \tilde{Y}$, define $h(y)$ to be the unique point in
$\cap\{\overline{U}^{\cA X}: U \in \cO_y\}$. If $f \in A(X,E)$ and
$h(y) \notin \overline{C}(f)$, then there exists an open set $U$ in
$\cA X$ so that $\overline{C}(f)\subseteq U$ and $h(y) \notin
\overline{U}^{\cA X}$.  Thus $U \notin \cO_y$.  Hence $Tf(y) = 0$.

\begin{lem}\label{lem7}
The map $h :\tilde{Y} \to \cA X$ is continuous.
\end{lem}

\begin{proof}
Fix $y_0 \in \tilde{Y}$.  Let $U$ be an open neighborhood of $x_0 =
h(y_0)$ in $\cA X$.   Let $V$ be an open neighborhood of $x_0$ in
$\cA X$ such that $\overline{V}^{\cA X} \subseteq U$.  Take $P =
i^{-1}(\cA X\bs U)$ and $Q = i^{-1}(V)$.  Then $\overline{i(P)}^{\cA X}\cap \overline{i(Q)}^{\cA X} = \emptyset$.  Hence there exists $\vp \in
A(X)$ so that $\vp(P) \subseteq \{0\}$ and $\vp(Q)\subseteq \{1\}$.
Let $f \in A(X,E)$ be such that $Tf(y_0) \neq 0$.  If $x \in Q$,
then $(f-\vp f)(x) = 0$.  Hence $x_0 \notin \overline{C}(f-\vp f)$.
By the above, $T(f-\vp f)(y_0) = 0$.  So $T(\vp f)(y_0) \neq 0$.
Choose a neighborhood $W$ of $y_0$ in $Y$ such that $T(\vp f)(y)
\neq 0$ for all $y \in W$.  If $y\in W \cap \tilde{Y}$ and $h(y)
\notin\overline{U}^{\cA X}$, then $h(y) \notin \overline{C}(\vp f)$.
Hence $T(\vp f)(y) =  0$, a contradiction.  So we must have $h(W
\cap \tilde{Y}) \subseteq \overline{U}^{\cA X}$. This proves that $h$ is
continuous on $\tilde{Y}$.
\end{proof}

In the sequel, we will be concerned with operators that are continuous with respect to certain topologies.  A useful consequence of it is that the map $h$ maps $\tilde{Y}$ into $i(X)$.
A linear topology on $A(X,E)$ is said to be {\em compactly determined} if for every neighborhood $U$ of $0$ in $A(X,E)$, there is a compact subset $K$ of $X$ such that $f \in U$ for all $f \in A(X,E)$ with $f_{|K} = 0$.

\begin{prop}\label{prop8}
Suppose that $T: A(X,E)\to A(Y,F)$ is linear and separating.  Assume that $T$ is continuous with respect to a compactly determined linear topology on $A(X,E)$ and a linear topology on $A(Y,F)$ that is stronger than the topology of pointwise convergence.  Let $h: \tilde{Y} \to \cA X$ be the support map.  Then $h(\tilde{Y}) \subseteq i(X)$.
\end{prop}

\begin{proof}
Suppose on the contrary that there exists $y_0 \in \tilde{Y}$ such that $h(y_0) = x_0 \notin i(X)$. There exists $f_0 \in A(X,E)$ such that $Tf_0(y_0) \neq 0$. By assumption, there is a neighborhood $V$ of $0$ in $A(Y,F)$ such that $\|g(y_0)\| <\|Tf_0(y_0)\|$ for all $g \in V$.
By continuity of $T$, there is a neighborhood $U$ of $0$ in $A(X,E)$ so that $TU \subseteq V$.  Since the topology on $A(X,E)$ is compactly determined, there exists a compact set $K$ in $X$ so that $f \in U$ for all $f \in A(X,E)$ with $f_{|K} = 0$. Now, $x_0$ does not lie within the compact set $i(K)$.  Hence there is an open neighborhood $\widetilde{Q}$ of $x_0$ in $\cA X$ such that $\overline{\widetilde{Q}}^{\cA X} \cap i(K) = \emptyset$.  Let $Q = i^{-1}(\widetilde{Q}) \subseteq X$.  Then
\[ \overline{i(K)}^{\cA X} \cap \overline{i(Q)}^{\cA X} \subseteq i(K) \cap \overline{\widetilde{Q}}^{\cA X} = \emptyset. \]
By Proposition \ref{prop2}, there exists $\vp \in A(X)$ such that $\vp(K) \subseteq \{0\}$ and $\vp(Q) \subseteq \{1\}$.
Consider the function $\vp f_0 \in A(X,E)$.    We claim that $x_0 \notin \overline{C}(f_0 - \vp f_0)$.  For otherwise, there exists $x \in \widetilde{Q}\cap i(C(f_0-\vp f_0)) \subseteq \widetilde{Q}\cap i(X)$.  Then $(f_0 - \vp f_0)(i^{-1}(x)) \neq 0$ and $x \in \widetilde{Q}$.  This contradicts the fact that $i^{-1}(x) \in i^{-1}(\widetilde{Q}) = Q$.  Thus the claim is verified.  By Proposition \ref{prop4}, $T(\vp f_0)(y_0) = Tf_0(y_0)$.  But since $\vp f_{0\,|K} = 0$, $T(\vp f_0) \in V$.  Consequently, $\|T(\vp f_0)(y_0)\| < \|Tf_0(y_0)\|$ by choice of $V$.  This contradiction completes the proof of the proposition.
\end{proof}

\section{Continuous separating maps between spaces of differentiable functions}

From hereon we will focus our attention on spaces of differentiable functions. Let $X$ and $Y$ be open subsets of Banach spaces $G$ and $H$ respectively.  Assume that $1 \leq p \leq \infty$, $0 \leq q \leq \infty$ and let $E$ and $F$ be Banach spaces.  The space  $C^p(X,E)$ consists of all functions $f:X \to E$ so that for all $k \in \N_p = \{i\in \N \cup\{0\}, i \leq p\}$, the (Fr\'{e}chet) derivative $D^kf$ is a continuous function from $X$ into $\cS^k(G,E)$, the space of all bounded symmetric $k$-linear maps from $G^k$ into $E$. (For notational convenience, we let $D^0f = f$ and $\cS^0(G,E) = E$.)
We will assume that there is a {\em bump function} $\vp \in C^p(G) = C^p(G,\R)$ so that $\vp(x) = 1$ if $\|x\| \leq 1/2$, $\vp(x) = 0$ if $\|x\| \geq 1$, and that $\sup_{x\in G}\|D^k\vp(x)\| < \infty$ for all $k \in \N_p$. Such is the case, for example, if $G$ is a subspace of $\ell^r$, $p \leq r < \infty$, or a subspace of $\ell^{2n}$ for some $n \in \N$. Refer to \cite{DGZ} for information on smoothness in Banach spaces.
The space $C^p(X)$ satisfies property ($*$) from \S 1 and,
with the assumed existence of bump functions as above, it separates points from closed sets.

Suppose that $T: C^p(X,E) \to C^q(Y,F)$ is  a linear separating map.  If we take $A(X,E) = C^p(X,E)$ and $A(X) = C^p(X)$, then Proposition \ref{prop4} applies to $T$.  Furthermore, it is clear that the set $\tilde{Y}$ given by Proposition \ref{prop4} is an open subset of $Y$ and thus an open subset of $H$.   In place of $T$, it suffices to consider the map from $C^p(X,E)$ into $C^q(\tilde{Y},F)$ given by $f\mapsto Tf_{|\tilde{Y}}$.  Without loss of generality, we may assume that $Y = \tilde{Y}$; equivalently, for all $y \in Y$, there exists $f\in C^p(X,E)$ so that $Tf(y) \neq 0$. We call such an operator $T$ {\em nowhere trivial}.

For each compact subset $K$ of $X$ and each  $k \in \N_p$, $\rho_{K,k}: C^p(X,E) \to \R$ given by
\begin{equation}\label{seminorms}
\rho_{K,k}(f) = \sup_{x\in K}\sup_{0\leq j\leq k}||D^jf(x)||
\end{equation}
defines a seminorm on $C^p(X,E)$.   The set of seminorms $\rho_{K,k}$, where $K$ is a compact subset of $X$ and $k \in \N_p$,
generates a compactly determined topology that is stronger than the topology of pointwise convergence.  From now on, endow $C^p(X,E)$ with this topology and $C^q(Y,F)$ with the corresponding topology.  In particular, assuming that $T: C^p(X,E) \to C^q(Y,F)$ is linear, separating and continuous, Proposition \ref{prop8} applies.  In this situation, we identify $i(X)$ with $X$ and consider the support map $h$ as a (continuous) map from $Y$ into $X$.

\begin{prop}\label{prop9}
Let $T: C^p(X,E) \to C^q(Y,F)$ be a separating, nowhere trivial and continuous linear operator and denote by $h$ the support map.
For any $y_0 \in Y$, there exist $r > 0$ and $k_0 \in \N_p$ such that  $B(y_0,r) \subseteq Y$ and for all $y \in B(y_0,r)$, $Tf(y) = 0$  for all $f\in C^p(X,E)$ with $D^mf(h(y)) = 0$ for all $0 \leq m \leq k_0$.
\end{prop}

\begin{proof}
If the proposition fails, there is a sequence $(y_k)$ in $Y$ converging to $y_0$ and a sequence of functions $(f_k)$ in $C^p(X,E)$ so that $D^mf_k(h(y_k)) = 0$, $0 \leq m \leq \min\{k,p\}$, but $Tf_k(y_k) \neq 0$. The set $K = \{y_k: k\in \N\cup\{0\}\}$ is compact.
By continuity of $T$, there exist a compact subset $L$ of $X$, $k_0 \in \N_p$, and $\delta > 0$ such that $\rho_{L,k_0}(f) \leq \delta$ implies $\rho_{K,0}(Tf) \leq 1$.
Consider the function $f = f_{k_0}\in C^p(X,E)$ and set $x_0 = h(y_{k_0})$.  Let $\vp \in C^p(G)$ be the bump function described above and let $C_i = \sup_{x\in G}\|D^i\vp(x)\|< \infty$ for each $i$. Fix $\ep > 0$.
Choose $\eta >0$ so that
$\eta 2^{k_0}\max_{0 \leq i \leq k_0} C_i \leq \delta\ep$.
Since $D^mf(x_0) = 0$, $0 \leq m \leq k_0$,
\[ \lim_{x\to x_0}\frac{\|D^if(x)\|}{\|x-x_0\|^{k_0-i}} = 0\]
for $0 \leq i \leq k_0$.  Choose $n \in \N$ so that
$\|D^if(x)\| \leq \eta\|x-x_0\|^{k_0-i}$
if $\|x-x_0\| \leq 1/n$ and $0 \leq i \leq k_0$.
Set $g(x) = \vp(n(x-x_0))f(x)$ for $x \in X$.  Then $g \in C^p(X,E)$ and $x_0 \notin \overline{C}(g-f)$.  By Proposition \ref{prop4}, $Tf(y_{k_0}) = Tg(y_{k_0})$.
Note that $g(x) =0$ for all $x \in X$ with $\|x-x_0\|\geq 1/n$.  Thus $D^mg(x) = 0$ for all $x\in X$, $\|x-x_0\|> 1/n$.
On the other hand, for any $x \in X$ with $\|x-x_0\|\leq 1/n$ and any $0 \leq m \leq k_0$,
\begin{align*}
\|D^mg(x)\| &\leq \sum^m_{i=0}n^i\binom{m}{i}\|D^i\vp(n(x-x_0))\|\|D^{m-i}f(x)\|\\
&
\leq \sum^m_{i=0}n^i\binom{m}{i}C_i\|D^{m-i}f(x)\|\\
&\leq \sum^m_{i=0}n^i\binom{m}{i}C_i\eta\|x-x_0\|^{k_0-m+i}\\
&\leq \eta 2^m\max_{0 \leq i \leq m} C_i \leq \delta\ep.
\end{align*}
By choice of $\delta$, it follows that $\rho_{K,0}(Tg) \leq \ep$. In particular, $\|Tf(y_{k_0})\|= \|Tg(y_{k_0})\| \leq \ep$. Hence $Tf(y_{k_0}) = 0$, yielding a contradiction.
\end{proof}

\noindent{\bf Remark}. By considering the seminorm $\rho_{K,j}$ in place of $\rho_{K,0}$, the proof shows that if $D^mf(h(y))=0$ for all $m \in \N_p$, then $D^j(Tf)(y) = 0$ for all $j \in \N_q$.

\bigskip

We now obtain a representation theorem for continuous linear separating maps between spaces of differentiable functions.
For any $e \in G$, let $e^k$ be the element $(e,\dots,e)$ ($k$ components) in $G^k$.

\begin{thm}\label{thm10}
Let $T: C^p(X,E) \to C^q(Y,F)$ be a separating, nowhere trivial and
continuous linear operator and denote by $h$ the support map.  There exist functions $\Phi_k: Y \times
\cS^k(G,E)\to F$, $k \in \N_p$, so that for all
$f \in C^p(X,E)$ and all $y \in Y$,
\begin{equation}\label{eq1.5}
Tf(y) = \sum_k\Phi_k(y,D^kf(h(y))).
\end{equation}
For each $y \in Y$, $\Phi_k(y,\cdot):
\cS^k(G,E) \to F$ is a bounded linear operator and, for each $S \in
\cS^k(G,E)$, $\Phi_k(\cdot,S): Y\to F$ is continuous on $Y$.
Furthermore, for every $y_0 \in Y$,  there exist $r > 0$ and  $k_0\in\N_p$ such that $\Phi_k(y,\cdot) = 0$ for all
$k > k_0$ and all $y \in B(y_0,r)$, and that $\{\Phi_k(y,\cdot): y \in B(y_0,r)\}$ is uniformly bounded for each $k \leq k_0$.
\end{thm}

\begin{proof}
For any $y \in Y$ and any $S\in \cS^k(G,E)$, $k \in \N_p$,  define the function $f_{y,S}: X \to E$ by $f_{y,S}(x) =
S[(x-h(y))^k]$. (If $k = 0$, we take the right hand side to mean
simply $S$.) Let $\Phi_k: Y \times \cS^k(G,E) \to F$ be given by
$\Phi_k(y,S) = Tf_{y,S}(y)/k!$.
For $y_0 \in Y$,
let $r_1 > 0$ and $k_0\in \N_p$ be given by Proposition \ref{prop9}.
Let $\ol{y} \in B(y_0,r_1)$ and $\ol{x} = h(\ol{y})$. If $k > k_0$, then, for any $S$, $D^mf_{\ol{y},S}(\ol{x}) =
0$, $0 \leq m\leq k_0$. Hence $\Phi_k(\ol{y},S) = 0$ by Proposition \ref{prop9}.
For any $f
\in C^p(X,E)$, consider the function
\[ P(x) = \sum^{k_0}_{k=0}\frac{D^kf(\ol{x})(x-\ol{x})^k}{k!}.\]
Then $P \in C^p(X,E)$ and $D^m(f-P)(\ol{x}) = 0$ for $0 \leq m\leq k_0$.
By choice of $k_0$, we have
\[ Tf(\ol{y}) = TP(\ol{y}) = \sum^{k_0}_{k=0}\Phi_k(\ol{y},D^kf(\ol{x})).\]
This gives rise to the representation (\ref{eq1.5}) since $\Phi_k(\overline{y},\cdot) = 0$ for all $k >k_0$.
The linearity of
$\Phi_k(y,\cdot)$ is clear from the definition and the linearity of
$T$. By direct computation, for any $y \in Y$, $S \in \cS^k(G,E)$
and $x \in X$,
\[ D^mf_{y,S}(x)(a_1,\dots,a_m) = \frac{k!}{(k-m)!}S((x-h(y))^{k-m},a_1,\dots,a_m)\]
for all $a_1,\dots,a_m \in G$ if $m \leq k$, and $D^mf_{y,S} = 0$ if $m > k$.
Let $L$ be a compact subset of $Y$. By continuity of $T$, there exist a compact $K \subseteq X$, $k \in \N_p$ and $C < \infty$ so that $\rho_{L,0}(Tf) \leq C \rho_{K,k}(f)$ for all $f \in C^p(X,E)$. From the expression for $D^mf_{y,S}$, we see that
\[ \sup\{\rho_{K,k}(f_{y,S}):y \in L, S\in \cS^k(G,E), \|S\| \leq 1\} < \infty.\]
It follows that
\[  \sup\{\|Tf_{y,S}(y)\|: y\in L, \|S\|\leq 1\} \leq \sup\{\rho_{L,0}(Tf_{y,S}): y\in L, \|S\|\leq 1\} < \infty.\]
This shows that each $\Phi_k(y,\cdot)$ is a bounded linear operator and that $\{\Phi_k(y,\cdot):y\in L\}$ is uniformly bounded on any compact subset $L$ of $Y$. If there exists $y_0 \in Y$ such that $\{\Phi_k(y,\cdot): y \in B(y_0,r)\}$ is not uniformly bounded for any $r > 0$, then there is a sequence $(y_n)$ converging to $y_0$ so that $(\Phi_k(y_n,\cdot))$ is unbounded. This contradicts the above since the set $\{(y_n)\}\cup\{y_0\}$ is compact in $Y$.
Similarly, if $S$ is bounded, it follows easily that if $(y_n)$
converges to $y$ in $Y$, then $(D^mf_{y_n,S})$ converges uniformly
to $D^mf_{y,S}$ on any compact subset $K$ of $X$ for any $m$.  Thus $(f_{y_n,S})$ converges to
$f_{y,S}$ in $C^p(X,E)$.   Hence $(Tf_{y_n,S})$ converges uniformly to $Tf_{y,S}$ on the compact set $\{(y_n)\}\cup\{y\}$. Therefore, $(\Phi_k(y_n,S))$ converges to
$\Phi_k(y,S)$.  This establishes the continuity of
$\Phi_k(\cdot,S)$.
\end{proof}

Theorem \ref{thm10} is not the last word since there is no mention of differentiability of $h$ or $\Phi_k$.  In the next few sections, we will see that in some cases (depending on $p$ and $q$), weak compactness of $T$ leads to degeneracy (constancy on connected components) of the support map.

\section{Degeneracy of the support map, $q > p$.}

This section and the next two are at the heart of the paper.  From hereon, $T: C^p(X,E) \to C^q(Y,F)$ will always be a separating, nowhere trivial and continuous linear operator, having a representation given by Theorem 10. The support map $h$ is implicitly determined by the operator $T$ via the representation given in Theorem \ref{thm10}.  We want to extract information on $h$ under various assumptions on $T$.  We begin with the simplest case in this section; namely, when $q > p$. Recall that we assume the existence of a bump function $\vp$ with bounded derivatives on the Banach space $G$ containing $X$.  Let $C_j = \sup_{x\in G}\|D^j\vp(x)\|$.

\begin{lem}\label{lem11}
Let $(x_n)$ be a sequence in $X$ converging to a point $x_0 \in X$.
Suppose that $\|x_n - x_0\| = r_n$, and that $0 < 3r_{n+1} < r_n$
for all $n \in \N$. Set $\vp_n(x) = \vp(\frac{2}{r_n}(x-x_n))$.  Let $(f_n)$ be a
sequence in $C^p(X,E)$ and let
\[ \eta_{nk} = \sum^k_{j=0}\binom{k}{j}\frac{C_j}{r^j_n}\sup_{x\in
B(x_n,r_n/2)}\|D^{k-j}f_n(x)\|\]
for all $n\in \N$ and all $k \in \N_p$.
If
\begin{align}
\lim_{n\to\infty}\frac{\eta_{nk}}{r_n}= 0,  &\quad 0\leq k < p, \label{eq1}\\
\intertext{and}
\lim_{n\to\infty}\eta_{np}= 0,  &\quad \text{in case $p < \infty$},
\label{eq2}
\end{align}
then the pointwise sum $f = \sum \vp_nf_n$
belongs to $C^p(X,E)$, and $D^kf(x_0) = 0$, $k \in \N_p$.
\end{lem}

\begin{proof}
Since the supports of the functions $\vp_nf_n$ tend toward the point $x_0$, it is clear that $f$ is $C^p$ on $X \bs \{x_0\}$. We first prove by induction that $D^kf(x_0) = 0$ for all $k \in \N_p$.
The case $k = 0$ is trivial.  Assume that the claim has been verified for some $k$, $0 \leq k < p$. If $x \notin \cup B(x_n,r_n/2)$, then $D^kf(x) = 0$.  Suppose that $x\in B(x_n,r_n/2)$ for some $n$.  Then
\begin{align*}
\|D^kf(x) - D^kf(x_0)\| & = \|D^kf(x)\| = \|D^k(\vp_nf_n)(x)\|\\
&\leq \sum^k_{j=0}\binom{k}{j}\|D^j\vp_n(x)\|\, \|D^{k-j}f_n(x)\|\\
&\leq \sum^k_{j=0}\binom{k}{j}\frac{2^jC_j}{r^j_n}\|D^{k-j}f_n(x)\| \leq 2^k\eta_{nk}.
\end{align*}
Also, $\|x-x_0\| \geq \frac{r_n}{2}$ for all $x\in B(x_n,r_n/2)$.  It follows from condition (\ref{eq1}) that $D^{k+1}f(x_0) = 0$. If $p < \infty$, we need to show that $D^pf$ is continuous at $x_0$.  But the same computation shows that
$\|D^pf(x) - D^pf(x_0)\| \leq 2^p\eta_{np}$ if $x \in B(x_n,r_n/2)$ for some $n$ and $0$ otherwise. The continuity of $D^pf$ at $x_0$ follows from condition (\ref{eq2}).
\end{proof}

\begin{thm}\label{thm12}
Assume that $1\leq p < \infty$ and that
at $y_0\in Y$, $\Phi_i(y_0,\cdot) \neq 0$ for some $i \leq p$. For all $q' \in \N_q$, there exist $r > 0$ and $C < \infty$ so that $\|h(y)-h(y_0)\|^{p-i} \leq C\|y-y_0\|^{q'}$ for all $y \in B(y_0,r)$.
\end{thm}

\begin{proof}
Otherwise, there are a sequence $(y_n)$ in $Y\bs\{y_0\}$ converging to $y_0$ and $q' \in \N_q$ so that $(\|h(y_n)-h(y_0)\|^{p-i}/\|y_n-y_0\|^{q'})$ diverges to $\infty$.
Hence it is possible to choose a positive sequence $(c_n)$ so that $c_n/\|h(y_n)-h(y_0)\|^{p-i} \to 0$ and $c_n/\|y_n-y_0\|^{q'}\to \infty$.  Set $x_n = h(y_n)$ and $x_0 = h(y_0)$. By using a subsequence if necessary, we may further assume that $0 < 3r_{n+1} < r_n$, where $r_n = \|x_n-x_0\|$. Let $S \in \cS^i(G,E)$ be such that $\Phi_i(y_0,S) \neq 0$.
Define $f_n(x) = c_nS(x-x_n)^i$ and $\vp_n$ as in Lemma \ref{lem11}. By choice of $c_n$, $f = \sum \vp_nf_n$ lies in $C^p(X,E)$ and $D^jf(x_0) = 0$ for all $j \in \N_p$. Thus $Tf$ is $C^q$ and, {\em a fortiori}, lies in $C^{q'}$. Furthermore, by the Remark following Proposition \ref{prop9}, $D^j(Tf)(y_0) = 0$ for $0 \leq j \leq q'$.  Hence $\lim_{y\to y_0}Tf(y)/\|y-y_0\|^{q'} = 0$. Since $Tf(y_n) = c_ni!\Phi_i(y_n,S)$ and $(\Phi_i(y_n,S))$ converges to $\Phi_i(y_0,S) \neq 0$, we have $c_n/\|y_n-y_0\|^{q'} \to 0$, contrary to the choice of $c_n$.
\end{proof}

\begin{cor}\label{cor13}
Let $T: C^p(X,E)\to C^q(Y,F)$ be a separating, nowhere trivial continuous linear operator. If $q > p$, then there is a partition of $Y$ into clopen subsets $(Y_\al)$ so that the support map is constant on each $Y_\al$.
\end{cor}

\begin{proof}
Choose $q' \in \N_q$ so that $q' > p$. Since $T$ is nowhere trivial, for any $y_0\in Y$, there exists $i \in \N_p$ so that $\Phi_i(y_0,\cdot)\neq 0$. Observe that $q' > p-i$. By Theorem \ref{thm12},
\[ \limsup_{y\to y_0}\frac{\|h(y)-h(y_0)\|}{\|y-y_0\|} \leq C\lim_{y\to y_0}\|y-y_0\|^{\frac{q'}{p-i} - 1} = 0.\]
Thus $Dh(y) = 0$ for all $y \in Y$ and the conclusion of the corollary follows.
\end{proof}

\section{Degeneracy of the support map, $p = q = \infty$}

Certainly, for different arrangements of $p$ and $q$, and without additional assumptions, degeneracy of the support map will not result. Recall that a set $W$ in $C^q(Y,F)$ is said to be {\em bounded} if it is absorbed by every open neighborhood of $0$; equivalently, if each seminorm $\rho_{L,\ell}$ (see (\ref{seminorms})) is bounded on $W$.
In this section, we show that when $p = q = \infty$ and $T$ maps some open set onto a bounded set, then the support map is degenerate.  The condition on $T$ is satisfied in particular when it is compact or weakly compact.

\begin{lem}\label{lem14}
For any $\ell \in \N$, there exists $c< \infty$ so that given $a \in X$, $0 < \ep < 1$ with $B(a,\ep) \subseteq X$ and $f\in C^\infty(X,E)$ with $D^jf(a) = 0$, $0\leq j \leq \ell$, there exists $g \in C^\infty(X,E)$ such that $g = f$ on $B(a,\ep)$ and
\[ \max_{0\leq j\leq \ell}\sup_{x\in X}\|D^jg(x)\| \leq c \sup_{x\in B(a,2\ep)}\|D^\ell f(x)\|.\]
\end{lem}

\begin{proof}
By translation, we may assume that $a = 0$. As usual, let $\vp$ denote the assumed bump function and set $c = 2^\ell\max_{0\leq k\leq \ell}\sup_{x\in G}\|D^k\vp(x)\|$. Given $\ep > 0$ with $B(0,\ep)\subseteq X$, consider the function $g: X \to E$ defined by $g(x) = \vp((2\ep)^{-1}x)f(x)$. Obviously, $g \in C^\infty(X,E)$ and $g = f$ on $B(0, \ep)$. Since $g(x) = 0$ when $\|x\| > 2\ep$, $D^jg(x) = 0$ for $\|x\| > 2\ep$ and any $j \in \N\cup\{0\}$. If $x\in X$ and $\|x\| \leq 2\ep$, by Taylor's Theorem (see, e.g., \cite{L}) and using the fact that $D^jf(0) = 0$, $0 \leq j \leq \ell$, we have, for $0 \leq k \leq \ell$,
\begin{align*}
\|D^kf(x)\| &\leq \int^1_0\frac{(1-t)^{\ell-k-1}}{(\ell-k-1)!}\|D^\ell f(tx)\|\|x\|^{\ell-k}\,dt \\
&\leq \sup_{z\in B(0,2\ep)}\|D^\ell f(z)\|\frac{(2\ep)^{\ell-k}}{(\ell-k)!}.
\end{align*}
Thus, for $0\leq j \leq \ell$,
\begin{align*}
\|D^jg(x)\| &\leq \sum^j_{k=0}\binom{j}{k}\bigl(\frac{1}{2\ep}\bigr)^{j-k}\|D^{j-k}\vp\bigl(\frac{x}{2\ep}\bigr)\|\,\|D^kf(x)\|\\
& \leq \sum^j_{k=0}\binom{j}{k}(2\ep)^{\ell-j}\|D^{j-k}\vp\bigl(\frac{x}{2\ep}\bigr)\|\sup_{z\in B(0,2\ep)}\|D^\ell f(z)\|\\
&\leq c\sup_{z\in B(0,2\ep)}\|D^\ell f(z)\|.
\end{align*}
\end{proof}

For the remainder of the section, fix a separating, nowhere trivial linear operator $T: C^\infty(X,E) \to C^\infty(Y,F)$ which maps some open set in $C^\infty(X,E)$ onto a bounded set in $C^\infty(Y,F)$. In particular, $T$ is continuous. Given $y_0\in Y$, we may find by applying Theorem \ref{thm10} $k_0 \in \N$ and $r> 0$ so that
\[ Tf(y) = \sum^{k_0}_{k=0}\Phi_k(y,D^kf(h(y))),\ y \in B(y_0,r)\subseteq Y, \ f\in C^\infty(X,E).\]
Moreover, since $T$ maps an open set onto a bounded set, there are a compact set $L \subseteq X$ and $\ell \in \N$ such that for all $m \in \N\cup\{0\}$ and every compact $K \subseteq Y$, there exists $C = C(K,m) < \infty$ so that
\begin{equation}\label{eq4}
\max_{y\in K}\|D^m(Tf)(y)\| \leq C \max_{0\leq j \leq \ell}\max_{x\in L}\|D^jf(x)\|.
\end{equation}
We may assume without loss of generality that $\ell \geq k_0$.

\begin{prop}\label{prop15}
If $f \in C^\infty(X,E)$ and $D^jf(h(y)) = 0$, $0 \leq j \leq \ell$, for some $y \in B(y_0,r)$ then $D^m(Tf)(y) = 0$, $m \in\N\cup\{0\}$.
\end{prop}

\begin{proof}
By Lemma \ref{lem14}, for any $\delta >0$, there exists $g \in C^\infty(X,E)$ so that $g = f$ on a neighborhood of $h(y)$ and
$\max_{0\leq j\leq \ell}\sup_{x\in X}\|D^jg(x)\| \leq \delta$. By the representation of $T$, $Tg = Tf$ on a neighborhood of $y$.
From (\ref{eq4}), we deduce that $\|D^m(Tf)(y)\| = \|D^m(Tg)(y)\|\leq C(\{y\},m)\delta$. The result follows since $\delta > 0$ is arbitrary.
\end{proof}

Denote $h(y_0)$ by $a$ and assume that $B(a,s) \subseteq X$ for some $0 < s < 1$.

\begin{prop}\label{prop16}
Let $K$ be a compact convex set in $B(y_0,r)$ containing the point $y_0$.
Suppose that $0 \leq j \leq k_0$, $S \in \cS^{j}(G,E)$ and $r_1, r_2,r_3, r_4$ are nonnegative integers so that $r_1 + r_3 = \ell+1$, $r_3+r_4 = j$ and $\al = r_2 + r_4 \leq k_0$. If $z\in K$, $b = h(z)\in B(a,s)$, and $x^*$ is a norm $1$ functional in the dual space $G^*$ of $G$ such that $x^*(b-a) = \|b-a\|$, define $f: X\to E$ by
\[ f(x) = (x^*(x-a))^{r_1}(x^*(x-b))^{r_2}S(x-a)^{r_3}(x-b)^{r_4}.
\]
Then there exists a finite constant $D = D(K)$ such that
\[ \|Tf(z)\| \leq D\|S\|\,\|b-a\|^{\al+1}\|z-y_0\|^{\ell+1}.\]
\end{prop}

\begin{proof}
Since $r_1 + r_3 > \ell$, $D^jf(a) = 0$, $0 \leq j \leq \ell$. Consider any $\ep$ so that $\|b-a\| < \ep < s$. By Lemma \ref{lem14}, there exists $g \in C^\infty(X,E)$ such that $g = f$ on $B(a,\ep)$ and
\[ \max_{0\leq j\leq \ell}\sup_{x\in X}\|D^jg(x)\| \leq c \sup_{x\in B(a,2\ep)}\|D^\ell f(x)\|.\]
Now, for $x\in B(a,2\ep)$,
\begin{multline*}
\|D^\ell f(x)\| \leq \ell!\sum\prod^4_{i=1}\binom{r_i}{s_i}\|x-a\|^{r_1-s_1}\|x-b\|^{r_2-s_2}\cdot
\\\|S\|\|\|x-a\|^{r_3-s_3}\|x-b\|^{r_4-s_4},
\end{multline*}
where the sum is taken over all integers $0 \leq s_i\leq r_i$, $1\leq i \leq 4$, with sum $\ell$. Using the estimates $\|x-a\|, \|x-b\|< 3\ep$ and
\[\sum\prod^4_{i=1}\binom{r_i}{s_i} \leq \prod^4_{i=1}\sum^{r_i}_{s_i = 0}\binom{r_i}{s_i} \leq 2^{\ell+1+\al}\leq  2^{2\ell+1},\]
we see that
\[ \|D^\ell f(x)\| \leq \ell!2^{2\ell+1}\|S\|(3\ep)^{\al+1}.\]
Take $C = C(K,\ell+1)$. From (\ref{eq4}) and the above, we obtain
\begin{align*}
\max_{y\in K}\|D^{\ell+1}(Tg)(y)\| &\leq C\max_{0\leq j\leq \ell}\max_{x\in L}\|D^jg(x)\| \\&
\leq Cc \sup_{x\in B(a,2\ep)}\|D^\ell f(x)\|\\
&\leq Cc\,\ell!2^{2\ell+1}3^{\ell+1}\|S\|\ep^{\al+1}.
\end{align*}
By Proposition \ref{prop15}, $D^m(Tg)(y_0) = 0$, $m \in \N\cup\{0\}$. Since $K$ is convex, by Taylor's Theorem,
\begin{multline*}
\|Tg(z)\| \leq \max_{y\in K}\|D^{\ell+1}(Tg)(y)\|\frac{\|z-y_0\|^{\ell+1}}{(\ell+1)!}\leq\\ \leq Cc\,2^{2\ell+1}3^{\ell+1}\|S\|\ep^{\al+1}\|z-y_0\|^{\ell+1}.
\end{multline*}
Observe that $Tf(z) = Tg(z)$ as $f = g$ on a neighborhood of $h(z) = b$. Since we may make $\ep$ as close to $\|b-a\|$ as we please, the result follows.
\end{proof}

\begin{prop}\label{prop17}
In the notation of Proposition \ref{prop16}, let
$W$ be the operator in $\cS^{\al}(G,E)$ determined by
\[ Ww^{\al} = (x^*(b-a))^{r_1}(x^*(w))^{r_2}S(b-a)^{r_3}w^{r_4}.\]
Then
\[ \al!\|\Phi_\al(z, W)\| \leq 2^{(k_0-\al)(\ell+2)} D\|S\|\,\|b-a\|^{\al +1}\|z-y_0\|^{\ell+1}.\]
\end{prop}

\begin{proof}
We proceed by induction on $\al$, beginning with $\al = k_0$ and ending at $\al = 0$. Suppose that $\al = k_0$. Let $f$ be the function defined in Proposition \ref{prop16}. Then
\[ \|\sum^{k_0}_{k=0}\Phi_k(z, D^kf(b))\| = \|Tf(z)\| \leq D\|S\|\,\|b-a\|^{\al+1}\|z-y_0\|^{\ell+1}.\]
Since $r_2 + r_4 = \al = k_0$, $D^kf(b) = 0$ if $k < k_0$, and $D^{k_0}f(b)= {k_0}!W$.
Hence
\[ k_0!\|\Phi_{k_0}(z,W)\| \leq D\|S\|\,\|b-a\|^{\al +1}\|z-y_0\|^{\ell+1}.\]

Now suppose that $0 \leq \al < k_0$, and the proposition has been proved for any $\beta$, $\al <\beta \leq k_0$.
Let $f$ be the function defined in Proposition \ref{prop16}, with the parameter $\al$. Then $D^kf(b) = 0$ if $k < \al$.
Hence
\[ \|\sum^{k_0}_{k=\al}\Phi_k(z, D^kf(b))\| = \|Tf(z)\| \leq D\|S\|\,\|b-a\|^{\al+1}\|z-y_0\|^{\ell+1}.\]
For $\al \leq k \leq k_0$,
\[ D^kf(b) = k!\sum\binom{r_1}{s_1}\binom{r_3}{s_3}W_{s_1,s_3},\]
where the sum is taken over all integers $0\leq s_1\leq r_1$, $0 \leq s_3 \leq r_3$, such that $s_1 + s_3 = k-\al$, and, for each such pair,
\[ W_{s_1,s_3}w^k = (x^*(b-a))^{r_1-s_1}(x^*(w))^{s_1+r_2}S(b-a)^{r_3-s_3}w^{s_3+r_4}.\]
If $k_0 \geq k > \al$, the inductive hypothesis gives
\[ k!\|\Phi_k(z,(x^*(b-a))^{k-\al}W_{s_1,s_3})\| \leq 2^{(k_0-k)(\ell+2)}D\|S\|\,\|b-a\|^{k +1}\|z-y_0\|^{\ell+1}.\]
Since $x^*(b-a) = \|b-a\|$ and $\sum\binom{r_1}{s_1}\binom{r_3}{s_3} \leq 2^{r_1 + r_3} = 2^{\ell+1}$, we get that
\[  \|\Phi_k(z, D^kf(b))\| \leq 2^{\ell+1+(k_0-k)(\ell+2)} D\|S\|\,\|b-a\|^{\al+1}\|z-y_0\|^{\ell+1}.\]
Then
\begin{align*}
\al!\|\Phi_\al(z,W)\| &= \|\Phi_\al(z,D^\al f(b))\| \\
&\leq  \sum^{k_0}_{k=\al+1}\|\Phi_k(z, D^kf(b))\| + D\|S\|\,\|b-a\|^{\al+1}\|z-y_0\|^{\ell+1}\\
&\leq [\sum^{k_0}_{k=\al+1}2^{\ell+1+(k_0-k)(\ell+2)}+1]\cdot D\|S\|\,\|b-a\|^{\al+1}\|z-y_0\|^{\ell+1}\\
&\leq 2^{(k_0-\al)(\ell+2)} D\|S\|\,\|b-a\|^{\al +1}\|z-y_0\|^{\ell+1},
\end{align*}
completing the induction.
\end{proof}

\begin{thm}\label{thm18}
Let $T: C^\infty(X,E)\to C^\infty(Y,F)$ be a separating, nowhere trivial linear operator which maps some open set in $C^\infty(X,E)$ onto a bounded set in $C^\infty(Y,F)$.  Then there is a partition of $Y$ into clopen subsets $(Y_\al)$ so that the support map is constant on each $Y_\al$.
\end{thm}

\begin{proof}
Let $y_0 \in Y$ and keep the foregoing notation. Since $T$ is nowhere trivial, there exist $0\leq j \leq k_0$ and $S\in \cS^j(G,E)$ so that $\Phi_j(y_0,S) \neq 0$. Choose $1 > r,s > 0$ so that $B(y_0,r)\subseteq Y$, $h(B(y_0,r))\subseteq B(h(y_0),s) \subseteq X$ and $\inf_{z\in B(y_0,r)}\|\Phi_j(z,S)\| = \delta > 0$. Take any $z_0 \in B(y_0,r)\bs\{y_0\}$ and let $K$ be the compact interval $[y_0,z_0]$.  Let $z \in K \bs\{y_0\}$, $b = h(z)$ and let $x^*$ be a norming functional for $b-a = b-h(y_0)$.
Apply Proposition \ref{prop17} with $r_1 = \ell+1$, $r_2 = r_3 = 0$
and $r_4 = j$. Note that $\al = j$ as well.  We obtain, with $D = D(K)$,
\[ j!|x^*(b-a)^{\ell+1}|\,\|\Phi_j(z,S)\| \leq 2^{(k_0-j)(\ell+2)}D\|S\|\,\|b-a\|^{j+1}\|z-y_0\|^{\ell+1}.\]
Hence
\[\delta j!\biggl(\frac{\|b-a\|}{\|z-y_0\|}\biggr)^{\ell-j} \leq 2^{(k_0-j)(\ell+2)}D\|S\|\|z-y_0\|^{j+1}.
\]
Note that $j \leq k_0 \leq \ell$. Since the right hand side tends to $0$ as $z$ tends to $y_0$ in $K$, $j \neq \ell$.  Thus $\ell - j > 0$.
Therefore,
\[ \lim_{\substack{z\to y_0\\z \in [y_0,z_0]}}\frac{h(z) - h(y_0)}{z-y_0} = 0.\]
That is, $h$ has directional derivative $0$ in every direction at any point $y_0 \in Y$.  It follows easily that $h$ must be constant on a neighborhood of every point in $Y$, from which the conclusion of the theorem is easily deduced.
\end{proof}

\noindent{\bf Example.} Suppose that $1 \leq q \leq p\leq \infty$
and that $q$ is finite. Then there may exist a separating, nowhere trivial linear operator $T: C^p(X,E)\to C^q(Y,F)$ which maps some open set in $C^p(X,E)$ onto a bounded set in $C^q(Y,F)$ whose support map is not locally constant.  Indeed, let $I = (0,1)$ and consider the map $T: C^p(\R) \to C^q(I)$ defined by $Tf = f_{|I}$.
Clearly, $T$ is a separating, nowhere trivial linear operator. Taking $L = [0,1]$, it maps the open set $\{f\in C^p(\R): \rho_{L,q}(f) < 1\}$  onto a bounded set in $C^q(I)$.  However, it is clear that the support map is $h(y) = y$ for all $y \in I$.

Furthermore, if $q < p$, then $T$ maps the open set $\{f\in C^p(\R): \rho_{L,q+1}(f) < 1\}$  onto a relatively compact set in $C^q(I)$.

\section{Degeneracy of the support map, $p = q < \infty$}

The example in the last section shows that the methods of the previous two sections cannot be used to show that the support map of $T$ is locally constant when $p = q < \infty$, even if $T$ maps an open set onto a bounded set.  However, we will show in this section that the support map is locally constant if the nowhere trivial, separating map $T: C^p(X,E) \to C^p(Y,F)$ is assumed to be weakly compact.

In this section, we assume that $1 \leq p < \infty$ and that $T: C^p(X,E) \to C^p(Y,F)$ is a weakly compact, separating, nowhere trivial linear map.  Denote by $h$ the support map of $T$. By Theorem \ref{thm10}, $T$ has the representation
\[  Tf(y) = \sum^p_{k=0}\Phi_k(y, D^kf(h(y))),\ y \in Y,\ f \in C^p(X,E).\]

\begin{lem}\label{lem19}
Suppose that there are $k_0$, $0 < k_0 \leq p$ and $y_0 \in Y$ such that $\Phi_{k_0}(y_0,\cdot) \neq 0$.  Then $h$ is constant on a neighborhood of $y_0$.
\end{lem}

\begin{proof}
By Theorem \ref{thm12}, there exist $r > 0$ and $C < \infty$ so that $\|h(y) - h(y_0)\|^{p-k_0} \leq \|y-y_0\|^p$ for all $y \in B(y_0,r)$. Hence $\lim_{y\to y_0}\|h(y)-h(y_0)\|/\|y-y_0\| = 0$; that is, $Dh(y_0) = 0$. Fix $S$ such that $\Phi_{k_0}(y_0,S) \neq 0$. By continuity of $\Phi_{k_0}(\cdot,S)$, $\Phi_{k_0}(y,\cdot) \neq 0$ for all $y$ in a neighborhood of $y_0$. Hence $Dh(y) = 0$ for all $y$ in a neighborhood of $y_0$.  The lemma follows.
\end{proof}

A subset $Y'$ of $Y$ is {\em perfect} if every point of $Y'$ is an accumulation point of $Y'$. If $f$ is a function from $Y'$ into a Banach space, then we may define the derivative of $f$ in the usual manner, with the relevant limit taken in $Y'$. Precisely, $f$ is differentiable at $y_0\in Y'$ if there exists a bounded linear operator $T$ so that
\[ \lim_{\substack{y\to y_0\\y \in Y'}}\frac{\|f(y) - f(y_0) - T(y-y_0)\|}{\|y-y_0\|} = 0.\]
The derivative operator $T$, if it exists, is denoted by $D_{Y'}f(y_0)$. Higher derivatives may be defined in the same manner. Usual rules of differentiation, in particular the chain rule and product and quotient rules for real-valued functions, still hold.

\begin{prop}\label{prop20}
Suppose that $Y'$ is a  perfect subset of $Y$ and that $\Phi_k(y,\cdot) = 0$ for all $y \in Y'$ and all $k$, $0 < k\leq p$. Then for all $x^* \in G^*$ and all $y_0 \in Y$, $D_{Y'}(x^*\circ h)(y_0) = 0$.
\end{prop}

\begin{proof}
By the assumptions, $T$ has the representation $Tf(y) = \Phi_0(y,f(h(y))$ for all $y \in Y'$. For any $y_0 \in Y'$, since $T$ is nowhere trivial, there exists $u \in E$ such that $\Phi_0(y_0,u)\neq 0$. Set $g(y) = \Phi_0(y,u), y \in Y'$. Then $g$ is the restriction to $Y'$ of a function in $C^p(Y,F)$. For any $\al \in C^p(X)$, $(\al\circ h) \cdot g = T(\al\cdot u)_{|Y'}$ is $p$-times continuously differentiable on $Y'$.
Choose $v^*\in F^*$ so that $v^*\circ g\neq 0$ on a neighborhood $O$ of $y_0$. Then $\al\circ h = v^*\circ T(\al\cdot u)/v^*\circ g$ on $O\cap Y'$ and hence $\al\circ h$ is $p$-times continuously differentiable there. Since $y_0 \in Y'$ is arbitrary, $\al\circ h$ is $p$-times continuously differentiable on $Y'$.

Let $x^* \in G^*$. By the preceding paragraph, $x^*\circ h$ is $p$-times continuously differentiable on $Y'$.  Consider a function $\al$ of the form $\beta\circ x^*$, where $\beta \in C^p(\R)$. We have, for any
$y \in Y'$ and $w \in H$,
\begin{equation}\label{eq5}
D_{Y'}^p(T(\al\cdot u))(y)w^p = \sum^p_{j=0}\binom{p}{j} D_{Y'}^j(\beta\circ x^*\circ h)(y)w^j\cdot D_{Y'}^{p-j}g(y)(w^{p-j}).
\end{equation}
Applying the chain rule to $\beta\circ(x^*\circ h)$ gives, for $1 \leq j \leq p$,
\begin{multline}\label{eq6}
D_{Y'}^j(\beta\circ x^*\circ h)(y)w^j = \sum^j_{i=1}\sum_{\substack{k_1+\cdots+k_i = j\\k_1,\dots, k_i \geq 1}} \binom{j}{k_1\ k_2\ \cdots k_i}\frac{1}{i!}\cdot\\ \beta^{(i)}((x^*\circ h)(y))\cdot D_{Y'}^{k_1}(x^*\circ h)(y)w^{k_1}\cdots D_{Y'}^{k_i}(x^*\circ h)(y)w^{k_i}.
\end{multline}
Fix $y_0 \in Y'$ and let $x_0 = h(y_0)$. Choose a sequence $(\beta_m)$ in $C^p(\R)$ so that
\begin{enumerate}
\item[(a)] $\max_{0\leq j \leq p}\sup_m\sup_{t\in \R}|\beta^{(j)}_m(t)| = M < \infty$,
\item[(b)] $\beta^{(p)}_m(t) = 1$ if $1/m \leq |t-x^*(x_0)| \leq 2$ and
\item[(c)] $\beta_m^{(j)}(x^*(x_0)) = 0$, $0\leq j \leq p$, for all $m$.
\end{enumerate}
Let $f_m = (\beta_m\circ x^*)\cdot u$.  Then $(f_m)$ is a bounded sequence in $C^p(X,E)$. Hence $(Tf_m)$ is relatively weakly compact in $C^p(Y,F)$. Let $(y_s)$ be a sequence in $Y'$ converging to $y_0$. The set $K = \{(y_s)\} \cup \{y_0\}$ is compact and thus the map $g \to D^pg_{|K}$ is a continuous linear operator from $C^p(Y,F)$ into $C(K,\cS^p(H,F))$; hence it is continuous with respect to the weak topologies on these spaces. Therefore, $(D^pT{f_m}_{|K})$ is relatively weakly compact in $C(K,\cS^p(H,F))$. To simplify notation, let us assume that $(D^pT{f_m}_{|K})$ converges weakly in $C(K,\cS^p(H,F))$. Then
\[  \lim_s\lim_m D^p(Tf_m)(y_s) = \lim_m\lim_s D^p(Tf_m)(y_s)  = \lim_mD^p(Tf_m)(y_0),\]
where the limits are taken with respect to the weak topology on $\cS^p(H,F)$.
Of course, the same equation holds if the general derivative $D$ is replaced by the restricted derivative $D_{Y'}$.
By (c) and equations (\ref{eq5}) and (\ref{eq6}), $D_{Y'}^p(Tf_m)(y_0) = 0$ for all $m$.  Hence $\lim_s\lim_m D_{Y'}^p(Tf_m)(y_s) = 0$ weakly. If $0 \leq i < p$, by (a) and (c),
\begin{align*}
|\beta_m^{(i)}((x^*\circ h)(y_s))|& =  |\beta_m^{(i)}((x^*\circ h)(y_s)) - \beta_m^{(i)}((x^*\circ h)(y_0))|\\ & \leq M|(x^*\circ h)(y_s) - (x^*\circ h)(y_0)|.
\end{align*}
Thus $\lim_s\sup_m|\beta_m^{(i)}((x^*\circ h)(y_s))| = 0$, $0 \leq i < p$. From this and (\ref{eq6}), it follows that $\lim_s\sup_m\|D_{Y'}^j(\beta_m\circ x^*\circ h)(y_s)\| = 0$ if $0 \leq j < p$, and that
\begin{multline*}
\lim_s\sup_m\sup_{\|w\|\leq 1}|D_{Y'}^p(\beta_m\circ x^*\circ h)(y_s)w^p -\\ \beta_m^{(p)}((x^*\circ h)(y_s))(D_{Y'}(x^*\circ h)(y_s)w)^p| = 0.
\end{multline*}
Therefore, from (\ref{eq5}), we obtain for $w \in H$ and with limits taken in the weak topology on $F$,
\begin{align*}
0 & = \lim_s\lim_m D_{Y'}^p(Tf_m)(y_s)w^p \\
& =\lim_s\lim_mD_{Y'}^p(\beta_m\circ x^*\circ h)(y_s)w^p\cdot g(y_s)\\
& = \lim_s\lim_m\beta_m^{(p)}((x^*\circ h)(y_s))(D_{Y'}(x^*\circ h)(y_s)w)^p\cdot g(y_s).
\end{align*}
Note that $\lim_sD_{Y'}(x^*\circ h)(y_s)w = D_{Y'}(x^*\circ h)(y_0)w$ and $\lim_sg(y_s) = g(y_0) \neq 0$. If $h(y_s) \neq x_0$, then $\lim_m\beta_m^{(p)}((x^*\circ h)(y_s)) = 1$ by (b).  Hence, if there is a sequence $(y_s)$ in $Y'$ converging to $y_0$ so that $h(y_s) \neq x_0$ for all $s$, then $D_{Y'}(x^*\circ h)(y_0) = 0$. Otherwise, $h$ is constant on a neighborhood of $y_0$ in $Y'$ and $D_{Y'}(x^*\circ h)(y_0) = 0$ also.
\end{proof}

\begin{cor}\label{cor21}
Suppose that $B(y_0,r) \subseteq Y$ and that $\Phi_k(y,\cdot) = 0$ for all $y \in B(y_0,r)$ and $0 < k\leq p$. Then $h$ is constant on $B(y_0,r)$.
\end{cor}

\begin{proof}
Take $Y'$ to be $B(y_0,r)$ in Proposition \ref{prop20}.  For all $y \in B(y_0,r)$ and all $x^* \in G^*$, $D(x^*\circ h)(y) = 0$. Thus, for all $x^*\in G^*$, $x^*\circ h$ is constant on $B(y_0,r)$. So $h$ is constant on $B(y_0,r)$.
\end{proof}

Let $U$ be the subset of $Y$ consisting of all points $y \in Y$ so that $h$ is constant on a neighborhood of $y$. Then the set $Y' = Y \bs U$ is a closed (in $Y$) perfect subset of $Y$. Indeed, $U$ is obviously open and hence $Y'$ is closed in $Y$. If $y_0$ is an isolated point in $Y'$, then there exists $r > 0$ so that $Dh(y) = 0$ for all $y \in B(y_0,r)\bs \{y_0\}$. Thus $h$ is constant on $B(y_0,r)\bs \{y_0\}$ and hence on $B(y_0,r)$ by continuity. This contradicts the fact that $y_0 \notin U$.

\begin{prop}\label{prop22}
Suppose that $y_0 \in Y'$. Then $D(x^*\circ h)(y_0) = 0$ for all $x^* \in G^*$.
\end{prop}

\begin{proof}
By Lemma \ref{lem19}, $\Phi_k(y,\cdot) = 0$ if $y \in Y'$ and $0 < k \leq p$. Fix $x^*\in G^*$. By Proposition \ref{prop20}, $D_{Y'}(x^*\circ h)(y) = 0$ for all $y \in Y'$. In particular, if $(y_n)$ is a sequence in $Y'\bs\{y_0\}$ converging to $y_0$, then
\[ \lim_n \frac{|(x^*\circ h)(y_n) - (x^*\circ h)(y_0)|}{\|y_n-y_0\|} = 0.\]
Now let $(y_n)$ be a sequence in $U$ converging to $y_0$. We may assume that all $y_n$ belong to an open ball centered at $y_0$ contained in $Y$.
For each $n$, set $t_n = \sup\{0\leq t < 1: (1-t)y_0 + ty_n \in Y'\}$ and $z_n = (1-t_n)y_0 + t_ny_n$. Since $Y'$ is closed in $Y$, $z_n \in Y'$. Also, the segment $(z_n,y_n]$ is contained in $U$ and hence $h$ is constant there. By continuity, $h(z_n) = h(y_n)$. Thus
\[ \lim_{\substack{n\\z_n \neq y_0}}\frac{|(x^*\circ h)(y_n) - (x^*\circ h)(y_0)|}{\|y_n-y_0\|} \leq \lim_{\substack{n\\z_n \neq y_0}}\frac{|(x^*\circ h)(z_n) - (x^*\circ h)(y_0)|}{\|z_n-y_0\|} = 0\]
since $D_{Y'}(x^*\circ h)(y_0) = 0$. While for those $n$ where $z_n = y_0$, $h(y_n) = h(y_0)$. This proves that $D(x^*\circ h)(y_0) = 0$ for all $x^*\in G^*$.
\end{proof}

\begin{thm}\label{thm23}
Let $1 \leq p < \infty$ and let $T: C^p(X,E) \to C^p(Y,F)$ be a weakly compact, separating, nowhere trivial linear operator. Then there is a partition of $Y$ into clopen subsets $(Y_\al)$ so that the support map is constant on each $Y_\al$.
\end{thm}

\begin{proof}
Fix $x^* \in G^*$. If $y \in U$, then $h$ is constant on a neighborhood of $y$. Thus $D(x^*\circ h)(y) = 0$. On the other hand, if $y \in Y'$, then $D(x^*\circ h)(y) = 0$ by Proposition \ref{prop22}. It follows that if a segment $[y_1,y_2]$ lies in $Y$, then $x^*(h(y_1)) = x^*(h(y_2))$. Since this holds for all $x^*\in G^*$, $h(y_1) = h(y_2)$.  Hence $h$ is constant on connected components of $Y$. As $Y$ is locally connected, the connected components of $Y$ are clopen in $Y$.
\end{proof}

\section{Compact and weakly compact mappings}

\begin{prop}\label{prop24}
Assume that $1 \leq p \leq \infty$ and $0 \leq q \leq \infty$.  Let $T: C^p(X,E) \to C^q(Y,F)$ be a nowhere trivial separating linear map. If $T$ maps an open set onto a bounded set, then the support map $h$ of $T$ maps $Y$ onto a relatively compact set in $X$. Moreover, in the representation of $T$ given by (\ref{eq1.5}) in Theorem \ref{thm10}, there exists $\ell \in \N_p$ such that $\Phi_k =0$ for all $k > \ell$.
\end{prop}

\begin{proof}
The map $T$ is continuous and hence we have the representation of $T$ given by Theorem \ref{thm10}. By the assumption, there exist $\ell \in \N_p$ and a compact set $L$ in $X$ so that the set $\{f: \rho_{L,\ell}(f)< 1\}$ is mapped onto a bounded set.
Suppose that there exists $y_0 \in Y$ with $h(y_0)\notin L$. Choose $r > 0$ so that $B(h(y_0),2r) \cap L = \emptyset$. Denoting by $\vp$ the assumed bump function on $G$, for any $k \in \N_p$ and any $S \in \cS^k(G,E)$, the function $f \in C^p(X,E)$ given by $f(x) = \vp(\frac{1}{r}(x-h(y_0))S(x-h(y_0))^k$ is identically $0$ on a neighborhood of $L$. Thus $\rho_{L,\ell}(mf) = 0$ for any $m \in \N$. Hence the sequence $(mTf)^\infty_{m=1}$ is bounded in the topology on $C^q(Y,F)$.  Therefore, $Tf = 0$. In particular, using the representation of $T$, $k!\Phi_k(y_0,S) = Tf(y_0) = 0$. This shows that $\Phi_k(y_0,\cdot) = 0$ for all $k \in \N_p$, contrary to the assumption that $T$ is non-trivial at the point $y_0$.  We have shown that $h(Y)$ is contained in the compact set $L$, and hence is relatively compact.

If $p < \infty$, then the second assertion of the proposition is obvious, since we may take $\ell = p$. If $p = \infty$, take $\ell$ as in the previous paragraph. Let $y_0 \in Y$. By the proof of Proposition \ref{prop15}, if $D^jf(h(y_0)) = 0$, $0\leq j \leq \ell$, then $D^m(Tf)(y_0) = 0$, $m \in \N$. In particular, $Tf(y_0) = 0$.  Given $k > \ell$ and $S \in \cS^k(G,E)$, consider the function $f(x) = S(x-h(y_0))^k$. By the above, $Tf(y_0)= 0$. Thus $k!\Phi_k(y_0,S) = 0$. This proves that $\Phi_k = 0$ for all $k > \ell$.
\end{proof}

Corollary \ref{cor13}, Theorems \ref{thm18} and \ref{thm23} and Proposition \ref{prop24} provide information on the support map of $T$ in various cases.  We now consider how (weak) compactness of $T$ informs on the maps $\Phi_k$ in the representation of $T$. The key observation is given in the next proposition.

\begin{prop}\label{prop25}
Define the map $J: C^q(Y,F) \to \prod_{K,m}C(K,\cS^m(H,F))$ by $Jf =
(D^mf_{|K})_{K,m}$, where the product is taken over all compact
subsets $K$ of $Y$ and all $m \in \N_q$. Then $J$
is a linear homeomorphic embedding with closed range.  In
particular, a subset $V$ of $C^q(Y,F)$ is relatively compact,
respectively, relatively weakly compact, if and only if $JV$ is
relatively compact, respectively, relatively weakly compact.
\end{prop}

\begin{proof}
It is clear that $J$ is a linear homeomorphic embedding.  If $J$ has
closed ($=$ weakly closed, because of convexity) range, then the final statement of the
proposition follows easily.

Suppose that $(f_\al)$ is a net in $C^q(Y,F)$ so that $(Jf_\al)$
converges. In particular, for each $y \in Y$, and each $m$,
$(D^mf_\al(y))_\al$ converges to some $g_m(y)$. We claim that the
function $g_m:Y \to \cS^m(H,F)$ is continuous.  Indeed, let $(y_k)$
be a sequence in $Y$ converging to some $y_0 \in Y$.  Then $K =
\{y_k: k \in \N \cup\{0\}\}$ is a compact subset of $Y$.  Hence
$(D^mf_{\al|K})_\al$ converges in $C(K,\cS^m(H,F))$.  The limit must
be $g_{m|K}$.  Thus, $g_{m|K}$ is continuous on $K$.  In particular,
$\lim_kg_m(y_k) = g_m(y_0)$.  This proves that $g_m$ is continuous
at any $y_0 \in Y$.  Next, we show that $Dg_m = g_{m+1}$ if $0 \leq
m < q$.  Fix $y_0 \in Y$ and let $V$ be a convex open neighborhood
of $y_0$ in $Y$. For all $y \in V$ and all $\al$,
\[ D^mf_\al(y) - D^mf_\al(y_0) = \int^1_0D^{m+1}f_\al((1-t)y_0+ty)(y-y_0)\,dt.\]
Since $(D^mf_\al)_\al$ and $(D^{m+1}f_\al)_\al$ converge uniformly
on the segment $[y_0,y]$ to $g_m$ and $g_{m+1}$ respectively, we
have
\[ g_m(y) - g_m(y_0) = \int^1_0g_{m+1}((1-t)y_0+ty)(y-y_0)\,dt.\]
Hence
\[ \frac{\|g_m(y) - g_m(y_0)-g_{m+1}(y_0)(y-y_0)\|}{\|y-y_0\|} \leq \sup_{s \in [y_0,y]}\|g_{m+1}(s)-g_{m+1}(y_0)\|.
\]
Because of the continuity of $g_{m+1}$, we find that $Dg_m(y_0) =
g_{m+1}(y_0)$, as desired.  It now follows easily that $D^mg_0 =
g_m$ and so $g_0 \in C^q(Y,F)$.  For each compact subset $K$ of $Y$
and each $m$, $(D^mf_{\al|K})_\al$ converges in $C(K,\cS^m(H,F))$ by
assumption, and the limit must be $g_{m|K} = D^mg_{0|K}$. Therefore,
$(Jf_\al)_\al$ converges to $Jg_0$.  This proves that $J$ has closed
range.
\end{proof}

\begin{thm}\label{thm26}
Assume that $1 \leq p \leq  q \leq \infty$. Let $T: C^p(X,E)\to C^q(Y,F)$ be a compact, separating, nowhere trivial linear operator with support map $h$. Then there exist $\ell\in \N_p$, functions $\Phi_k: Y\times \cS^k(G,E) \to F$, $0 \leq k \leq \ell$, a relatively compact set $(x_\al)$ in $X$ and a partition $(Y_\al)$ of $Y$ into clopen subsets of $Y$ so that
\begin{equation}\label{eq7}
Tf(y) = \sum^{\ell}_{k=0}\Phi_k(y, D^kf(x_\al))
\end{equation}
for all $f\in C^p(X,E)$ and all $y \in Y_\al$. For a given $y$, $\Phi_k(y,\cdot)$ is a bounded linear operator from $\cS^k(G,E)$ into $F$.  For each $S \in \cS^k(G,E)$, $\Phi_k(y,S)$ belongs to $C^q(Y,F)$ as a function of $y$.  Moreover, for any $k \leq \ell$,
\begin{enumerate}
\item[(a)] For any $y \in Y$ and $m \in \N_q$, the map $D_y^m\Phi_k(y,\cdot): \cS^k(G,E) \to \cS^m(H,F)$ is a compact operator ($D_y$ denotes the derivative with respect to $y$);
\item[(b)] For each $m \in \N_q$, the map $y\mapsto D_y^m\Phi_k(y,\cdot)$ from $Y$ into the space of bounded linear operators $L(\cS^k(G,E), \cS^m(H,F))$ is continuous with respect to the operator norm topology.
\end{enumerate}
Conversely, any operator $T$ satisfying all of the above is a compact, separating linear operator from $C^p(X,E)$ into $C^q(Y,F)$.
\end{thm}

\begin{proof}
Suppose that $T$ is a compact, separating, nowhere trivial linear operator. By Theorem \ref{thm10}, $T$ has the representation given by (\ref{eq1.5}). By Corollary \ref{cor13}, Theorem \ref{thm18} or Theorem \ref{thm23}, there is a partition $(Y_\al)$ of $Y$ into clopen subsets so that the support map $h$ is constant on each $Y_\al$. Denote the value of $h$ on $Y_\al$ by $x_\al$. By Proposition \ref{prop24}, the set $(x_\al)$ is relatively compact in $X$ and there exists $\ell \in \N_p$ such that $\Phi_k = 0$ for all $k > \ell$. Hence we obtain the representation (\ref{eq7}) of $T$.
Given $k\leq \ell$, $S\in \cS^k(G,E)$ and $\al$, let $f(x)=S(x-x_\al)^k$. Then $Tf(y) = k!\Phi_k(y,S)$ for all $y \in Y_\al$. Since $Y_\al$ is open, this proves that $\Phi_k(y,S)$ is $C^q$ as a function of $y \in Y_\al$. As this applies to all $\al$, $\Phi_k(y,S)$ belongs to $C^q(Y,F)$ as a function of $y$.

From compactness of $T$, we obtain $i \in \N_p$ and a compact set $L$ in $X$ so that $\{Tf: \rho_{L,i}(f)<1\}$ is relatively compact in $C^q(Y,F)$. By Proposition \ref{prop25}, $\{D^m(Tf)_{|K}: \rho_{L,i}(f)<1\}$ is relatively compact in $C(K,\cS^m(H,F))$ for each $m \in \N_q$ and each compact set $K$ in $Y$. By Arzel\`{a}-Ascoli's Theorem \cite[Theorem IV.6.7]{DS},  $\{D^m(Tf)(y): \rho_{L,i}(f)<1\}$ is relatively compact in $\cS^m(H,F)$ for each $y \in K$ and $\{D^m(Tf)_{|K}: \rho_{L,i}(f)<1\}$ is equicontinuous on $K$.
Given $k \leq \ell$ and $\al$, define $f_S(x) = S(x-x_\al)^k$ for all $S \in \cS^k(G,E)$. Since $L$ is compact, there exists $\ep > 0$ so that $\rho_{L,i}(f_S) < 1$ whenever $\|S\| \leq \ep$. Let $m \in \N_q$.
By direct computation, $D^m(Tf_S)(y) = k!D^m_y\Phi_k(y,S)$ for all $y \in Y_\al$.  Choose $K$ to be any singleton $\{y\}$ in $Y_\al$.
Then $\{k!D^m_y\Phi_k(y,S): \|S\| \leq \ep\}$ is relatively compact. Thus $D^m_y\Phi_k(y,\cdot)$ is a compact operator. On the other hand, choose $K$ to be any convergent sequence $(y_j)$ in $Y_\al$ together with the limit $y_0$. Since the set $\{D^m(Tf_S): \|S\| \leq \ep\}$ is equicontinuous on $K$,
\[ \lim_j\sup_{\|S\|\leq \ep}\|D^m_y\Phi_k(y_j,S) - D^m_y\Phi_k(y_0,S)\| = 0\]
This proves that the function $y\mapsto D^m_y\Phi_k(y,\cdot)$ is continuous with respect to the norm topology on $L(\cS^k(G,E), \cS^m(H,F))$.

In the converse direction, let $\ell$ be given by the statement of the theorem and $L$ be a compact set in $X$ containing $(x_\al)$. From conditions (a) and (b) and equation (\ref{eq7}), one easily deduces that for all $m \in \N_q$ and all compact sets $K$, $\{D^m(Tf)(y): \rho_{L,\ell}(f)<1\}$ is relatively compact in $\cS^m(H,F)$ for each $y \in K$ and $\{D^m(Tf)_{|K}: \rho_{L,\ell}(f)<1\}$ is equicontinuous on $K$. Hence $\{D^m(Tf)_{|K}: \rho_{L,\ell}(f)<1\}$ is relatively compact in $C(K,\cS^m(H,F))$ by
Arzel\`{a}-Ascoli's Theorem.
Relative compactness of the set  $\{Tf: \rho_{L,\ell}(f)<1\}$ follows from Proposition \ref{prop25}.
\end{proof}

Given a bounded linear operator $T$ mapping between Banach spaces, let $T^*$ denote its adjoint.  If $G_1$ and $G_2$ are Banach spaces, the {\em strong operator topology} on $L(G_1,G_2)$ is the topology generated by the seminorms $\rho_u(T) = \|Tu\|$, $u \in G_1$.

\begin{thm}\label{thm27}
Assume that $1 \leq p \leq q \leq \infty$. Let $T: C^p(X,E)\to C^q(Y,F)$ be a weakly compact, separating, nowhere trivial linear operator with support map $h$. Then there exist $\ell\in \N_p$, functions $\Phi_k: Y\times \cS^k(G,E) \to F$, $0 \leq k \leq \ell$, a relatively compact set $(x_\al)$ in $X$ and a partition $(Y_\al)$ of $Y$ into clopen subsets of $Y$ so that
\begin{equation*}
Tf(y) = \sum^{\ell}_{k=0}\Phi_k(y, D^kf(x_\al))
\end{equation*}
for all $f\in C^p(X,E)$ and all $y \in Y_\al$. For a given $y$, $\Phi_k(y,\cdot)$ is a bounded linear operator from $\cS^k(G,E)$ into $F$.  For each $S \in \cS^k(G,E)$, $\Phi_k(y,S)$ belongs to $C^q(Y,F)$ as a function of $y$.  Moreover, for any $k \leq \ell$,
\begin{enumerate}
\item[(c)] For any $y \in Y$ and $m \in \N_q$, the map $D_y^m\Phi_k(y,\cdot): \cS^k(G,E) \to \cS^m(H,F)$ is a weakly compact operator;
\item[(d)] For each $m \in \N_q$, the map $y\mapsto (D_y^m\Phi_k(y,\cdot))^{**}$ from $Y$ into the space of bounded linear operators $L(\cS^k(G,E)^{**}, \cS^m(H,F)^{**})$ is continuous with respect to the strong operator topology.
\end{enumerate}
Conversely, any operator $T$ satisfying all of the above is a weakly compact, separating linear operator from $C^p(X,E)$ into $C^q(Y,F)$.
\end{thm}

\begin{proof}
Following the proof of Theorem \ref{thm26}, it suffices to show that an operator $T$ of the form (\ref{eq7}) is weakly compact if and only if conditions (c) and (d) hold.
Let $J$ be the map defined in Proposition \ref{prop25}.  By that result, a set $V$ is relatively weakly compact in $C^q(Y,F)$ if and only if $JV$ is relatively weakly compact in $\prod_{K,m} C(K,\cS^m(H,F))$. By \cite[Theorem IV.4.3]{S}, the weak topology on $\prod_{K,m} C(K,\cS^m(H,F))$ is the product of the weak topologies on $C(K,\cS^m(H,F))$.  Thus, by Tychonoff's Theorem, $V$ is relatively weakly compact if and only if $\{D^mf_{|K}: f \in V\}$ is relatively weakly compact in $C(K,\cS^m(H,F))$ for every $m \in \N_q$, and every compact subset $K$ of $Y$.

Assume that $T$ is weakly compact.
There exist a compact subset ${L}$ of $X$ and $i \in \N_p$, so that$\{Tf: \rho_{L,i}(f)<1\}$ is relatively weakly compact.
Given $k \leq \ell$ and $\al$, define $f_S(x) = S(x-x_\al)^k$ for all $S \in \cS^k(G,E)$. Since $L$ is compact, there exists $\ep > 0$ so that $\rho_{L,i}(f_S) < 1$ whenever $\|S\| \leq \ep$.
Choose $K$ to be any singleton $\{y\}$ in $Y_\al$. Then
for each $m \in \N_q$,
\[ \{D^m(Tf_{S})(y): \|S\| \leq \ep\} = \{k!D^m_y\Phi_k(y,S): \|S\| \leq \ep\}\]
is relatively weakly compact in $C(\{y\}, \cS^m(H,F)) = \cS^m(H,F)$. Hence condition (c) holds. On the other hand, choose $K$ to be any convergent sequence $(y_j)$ in $Y_\al$ together with the limit $y_0$.  Let $m \in \N_q$ and denote by $B$ the closed unit ball of $(\cS^m(H,F))^*$, endowed with the weak$^*$ topology, which renders it into a compact Hausdorff space. The map $I: C(K, \cS^m(H,F)) \to C(K\times B)$ given by $Ig(y,\gamma) = \gamma(g(y))$ is a linear isometric embedding. Set $W = \{D^m(Tf_S)_{|K}: \|S\| \leq \ep\}$.
Then $IW$ is relatively weakly compact in $C(K\times B)$. By \cite[Theorem IV.6.14]{DS}, the closure of $IW$ in the pointwise topology is contained in $C(K\times B)$.  Denote by $R_y$ the operator $D^m_y\Phi_k(y,\cdot)$.
Suppose that $S \in \cS^k(G,E)^{**}$ with $\|S\| \leq \ep$. Define $g_{S}: K \times B \to \R$ by $g_{S}(y,\gamma) = (R_y^{**}S)(\gamma)$. There exists a net $(S_\beta)$ in $\cS^k(G,E)$, $\|S_\beta\| \leq \ep$, that converges to $S$ in the weak$^*$ topology.  Since $R_y(S_\beta) = D^m(Tf_{S_\beta/k!})(y)$, the functions $(y,\gamma)\mapsto \gamma(R_yS_\beta)$ belong to $IW \subseteq C(K\times B)$ and converge pointwise to $g_{S}$.  Thus $g_{S}$ is continuous on $K \times B$.  Therefore, $\lim_{j}g_{S}(y_j,\gamma) = g_{S}(y_0,\gamma)$ uniformly for $\gamma \in B$, which means that $\lim_jR^{**}_{y_j}S = R^{**}_{y_0}S$ in the norm of $(\cS^m(H,F))^{**}$. This proves condition (d).

In the converse direction, let $\ell$ be given by the statement of the theorem and $L$ be a compact set in $X$ containing $(x_\al)$.
We claim that $\{Tf: \rho_{L,\ell}(f) < 1\}$ is relatively weakly compact in $C^q(Y,F)$. By the discussion in the first paragraph of the proof and using the isometric embedding $I$ as above, it suffices to show for that each $m \in \N_q$, and each compact subset $K$ of some $Y_\al$, the set $\{ID^m(Tf)_{|K}: \rho_{L,\ell}(f) < 1 \}$ is relatively weakly compact in $C(K\times B)$. By  \cite[Theorem IV.6.14]{DS}, we only need to show that the pointwise closure of this set is contained in $C(K\times B)$.
Let $(f_\beta)$ be a net with $\rho_{L,\ell}(f_\beta) < 1$ for all $\beta$ and assume that $(ID^m(Tf_\beta)_{|K})_\beta$ converges pointwise to a function $g$ on $K \times B$.
For $0 \leq k \leq \ell$, $\|D^kf_\beta(x_\al)\| < 1$. Thus, we may assume that $(D^kf_\beta(x_\al))_\beta$ converges to some $S_k \in (\cS^k(G,E))^{**}$ in the weak$^*$ topology.
Denote by $R_{k,y}$ the operator $D^m_y\Phi_k(y,\cdot)$.
By condition (c), $R_{k,y}$ is a weakly compact operator for each $y$. Hence $R^{**}_{k,y}S_k \in \cS^m(H_,F)$ and $\lim_\beta R_{k,y}(D^kf_\beta(x_\al)) = R^{**}_{k,y}(S_k)$ weakly.
Therefore, for all $y \in K$ and all $\gamma \in B$,
\begin{align*}
g(y,\gamma) &= \lim_\beta\gamma(D^m(Tf_\beta)(y)) = \lim_\beta\sum^\ell_{k=0}\gamma(R_{k,y}(D^kf_\beta(x_\al)))\\ &= \sum^\ell_{k=0}\gamma(R_{k,y}^{**}(S_k)).
\end{align*}
It follows that, for each $y \in K$, $g(y,\cdot)$ is a continuous function on $B$.
By condition (d), if $(y_j)$ converges to $y$ in $K$, then $(\gamma(R_{k,y_j}^{**}S_k))_j$ converges  to $\gamma(R_{k,y}^{**}S_k)$
uniformly in $\gamma \in B$; that is, $g(y_j,\gamma)$ converges to $g(y,\gamma)$ uniformly in $\gamma \in B$.  Since, by the above, each $g(y_j,\cdot)$ is a continuous function on $B$, it follows easily that $g$ is a continuous function on $K \times B$.
\end{proof}

\section{The case $p > q$}

The example at the end of Section 4 tells us that when $p > q$, the support map of a compact, disjointness preserving linear operator $T: C^p(X,E) \to C^q(Y,F)$ need not be degenerate. Indeed, we shall see that the support map need not even be differentiable at all points of $Y$.  In this section, we will present the counterparts of Theorems \ref{thm26} and \ref{thm27} when $p > q$ and {\em $X$ is an open set in $\R^n$}.

As indicated, in this section, $X$ will be an open set in $\R^n$. A {\em multiindex} is an $n$-tuple of nonnegative integers $\lambda = (\lambda_1,\cdots, \lambda_n)$. The {\em order} of $\lambda$ is $|\lambda| = \sum^n_{i=1}\lambda_i$. If $x = (x_1,\dots, x_n) \in \R^n$ and $\lambda$ is a multiindex, let $x^\lambda = \prod^n_{i=1}x^{\lambda_i}$. Given a multiindex $\lambda$ and $u \in E$, let $S_{\lambda,u}$ be the operator in $\cS^{k}(\R^n,E)$, $k = |\lambda|$, so that $S_{\lambda,u}x^k = x^\lambda u$. The set of operators $\{S_{\lambda,u}: |\lambda| = k, u \in E\}$ span $\cS^k(\R^n,E)$. For the rest of the section, let  $T: C^p(X,E) \to C^q(Y,F)$ be given a continuous, separating, nowhere trivial linear operator. Denote by $h$ the support map. By Theorem \ref{thm10}, $T$ has the representation
\[ Tf(y) = \sum_k\Phi_k(y,D^kf(h(y)),\]
where the sum is locally finite.

\begin{prop}\label{prop28}
Suppose that $1 \leq i \leq n$. There is a dense open subset $U_i$ of $Y$ so that for all $y_0 \in U_i$, there exist $m_0 \geq 0$ and an open neighborhood $V$ of $y_0$ so that $\Phi_{|\lambda|}(y,S_{\lambda,u}) = 0$ if $y \in V$, $u \in E$ and  $\lambda_i > m_0$, and $\Phi_{|\mu|}(y_0,S_{\mu,u})\neq 0$ for some multiindex $\mu$ with $\mu_i = m_0$ and some $u\in E$.
\end{prop}

\begin{proof}
$Y$ can be covered by open subsets $W$, where for each $W$, there exists $k_0 = k_0(W) \in \N$ such that $\Phi_k(y,\cdot) = 0$ for all $k > k_0$ and all $y \in W$.  For $m \geq 0$, let
\begin{align*}
A_m & = \{y \in Y: \text{there exist } \lambda \text{ and } u \text{ so that } \lambda_i = m, \Phi_{|\lambda|}(y,S_{\lambda,u}) \neq 0\}\\
B_m & = \{y \in Y: \text{for all } \lambda \text{ and } u \text{ with } \lambda_i = m, \Phi_{|\lambda|}(y,S_{\lambda,u}) = 0\}.
\end{align*}
Clearly, $A_m$ is open and $A_m \cup B_m = Y$.
For each $W$ and $k_0$ as above,
\[ U_W =  W \cap [A_{k_0} \cup \bigcup^{k_0-1}_{m=0}(A_m \cap \inte B_{m+1}\cap \dots\cap \inte B_{k_0}) \cup \bigcap^{k_0}_{m=0}\inte B_m]\]
is a dense open set in $W$. If $y \in W \cap \bigcap^{k_0}_{m=0}\inte B_m$, then $\Phi_{|\lambda|}(y,S_{\lambda,u}) = 0$ whenever $\lambda_i \leq k_0$ and $u \in E$. Thus $\Phi_k(y,\cdot)= 0$ for all $k \leq k_0$. Since $y \in W$, $\Phi_k(y,\cdot) = 0$ for $k > k_0$ as well. This contradicts the fact that $T$ is nontrivial at $y$. Hence
\[ U_W =  W \cap [A_{k_0} \cup \bigcup^{k_0-1}_{m=0}(A_m \cap \inte B_{m+1}\cap \dots\cap \inte B_{k_0})].\]
The set $U_i = \cup_WU_W$
fulfills the conditions of the proposition.
\end{proof}

Let $U$ be the set $\cap^n_{i=1}U_i$.  Then $U$ is a dense open subset of $Y$.  Denote the $i$-th component of $h$ by $h_i$. Given multiindices $\lambda = (\lambda_1,\dots,\lambda_n)$ and $\nu = (\nu_1,\dots,\nu_n)$, we write $\lambda \leq \nu$ to mean that $\lambda_i \leq \nu_i$, $1\leq i \leq n$. In this case, we let the binomial coefficient $\binom{\nu}{\lambda}$ be $\prod^n_{i=1}\binom{\nu_i}{\lambda_i}$.

\begin{prop}\label{prop29}
For $1\leq i \leq n$, $h_i$ is $C^q$ on the set $U_i$.  Therefore, $h$ is $C^q$ on the dense open set $U$ of $Y$.
\end{prop}

\begin{proof}
Take $y_0 \in U_i$. We may assume that $i=1$ and $h(y_0) = 0$.
Let $V$ and $m_0$ be given by Proposition \ref{prop28} and let $\mu = (\mu_1,\dots,\mu_n)$, $u \in E$ and $v^*\in F^*$ be such that $\mu_1 = m_0$ and $v^*(\Phi_{|\mu|}(y_0,S_{\mu,u})) \neq 0$.
If $0 \leq j \leq m_0+1$, let $f_j: X \to E$ be the function $f_j(x) = x^{\nu_j}u$, where $\nu_j$ is the multiindex $(j,\mu_2,\dots,\mu_n)$.  Then $f_j \in C^p(X,E)$ and, for all $y \in V$,
\[ Tf_j(y) = \sum_{k=0}^{j\wedge m_0}\sum_{\substack{\lambda \leq \nu_j\\ \lambda_1=k}}|\lambda|!\binom{\nu_j}{\lambda}h(y)^{\nu_j-\lambda}\cdot\Phi_{|\lambda|}(y,S_{\lambda,u}) .\]
If $0 \leq j \leq m_0$, set
\[ g_j(y) = \sum_{\substack{\lambda \leq \nu_j\\ \lambda_1=j}}|\lambda|!\binom{\nu_j}{\lambda}h(y)^{\nu_j-\lambda}\cdot v^*(\Phi_{|\lambda|}(y,S_{\lambda,u})).\]
Then
\begin{align*}
\sum^j_{k=0}\binom{j}{k}\,g_{j-k}h_1^k &= v^*\circ Tf_j,\ 0\leq j\leq m_0, \text{ and} \\
\sum^{m_0+1}_{k=1}\binom{m_0+1}{k}\,g_{m_0+1-k}h_1^k &= v^*\circ Tf_{m_0+1}
\end{align*}
are $C^q$ functions on $V$.
Consider functions $F_j: \R^{m_0+2} \to \R$ given by
\[ F_j(t_0,\cdots, t_{m_0},s) = \sum^j_{k=0}\binom{j}{k}t_{j-k}s^k\]
for $0\leq j \leq m_0$ and $F_{m_0+1}(t_0,\cdots, t_{m_0},s) = \sum^{m_0+1}_{k=1}\binom{m_0+1}{k}t_{m_0+1-k}s^k$.
Now
\begin{align*} \sum^{m_0+1}_{j=0}(-1)^{j}&\binom{m_0+1}{j}s^{m_0+1-j}\partial_{t_i}F_j \\&= (-1)^{i}s^{m_0+1-i}\binom{m_0+1}{i}\sum^{m_0+1-i}_{j=0}\binom{m_0+1-i}{j}(-1)^{j}\\ &= (-1)^{i}s^{m_0+1}\binom{m_0+1}{i}(1-1)^{m_0+1-i} = 0,
\end{align*}
and
\begin{align*}
\sum^{m_0+1}_{j=0}&(-1)^{j}\binom{m_0+1}{j}s^{m_0+1-j}\partial_{s}F_j \\&= \sum^{m_0+1}_{j=1}\sum^j_{k=1}(-1)^j\binom{m_0+1}{j}\binom{j}{k}kt_{j-k}s^{m_0+k-j},
\intertext{which yields, upon substituting $i = j-k$ and switching the order of summation,}
&= \sum^{m_0}_{i=0}(-1)^{i}t_is^{m_0-i}\binom{m_0+1}{i}\bigl[\sum^{m_0+1-i}_{j=1}\binom{m_0-i+1}{j}j(-1)^j\bigr]\\
&= \sum^{m_0}_{i=0}(-1)^it_is^{m_0-i}\binom{m_0+1}{i}\bigl[\frac{d}{dt}(1-t)^{m_0-i+1}\bigr]_{|t=1} \\& = (-1)^{m_0}(m_0+1)t_{m_0}.
\end{align*}
It follows that the function $F = (F_0,\dots,F_{m_0+1})$ has a local $C^\infty$ inverse at any $(t_0,\dots,t_{m_0},s)$ with $t_{m_0} \neq 0$. Since $F(g_0,\dots,g_{m_0},h_1)$ is $C^q$ on $V$ and, because $h(y_0) = 0$,
\[ g_{m_0}(y_0) = |\mu|!\,v^*(\Phi_{|\mu|}(y_0,S_{\mu,u})) \neq 0,\]
the function $h_1$, in particular, must be $C^q$ on a neighborhood of $y_0$.
\end{proof}

The next example shows that $h$ need not be differentiable at all points of $Y$.

\bigskip

\noindent{\bf Example}.
Assume that $0 \leq q < p <\infty$. If $f \in C^p(\R)$, then it is easy to check that the function
\[ g(t) = \begin{cases}
          \sum^p_{i=0}\frac{t^i}{i!}f^{(i)}(0) &\text{ $t < 0$}\\
          f(t) &\text{ $t \geq 0$}
          \end{cases}\]
belongs to $C^p(\R)$. Consider the map $R: f \mapsto g$ as an operator from $C^p(\R)$ to $C^q(I)$, where $I = (-1,1)$. Then $R$ is a continuous, separating, nowhere trivial linear operator. It is even compact as it maps the open set $\{f\in C^p: \max_{k\leq p}\sup_{|t|\leq 1}|f^{(k)}(t)| < 1\}$ onto a compact subset of $C^q(I)$. However, the support map $h: I \to \R$ of $R$, given by  $h(t) = 0$ if $t < 0$ and $h(t) = t$ if $t \geq 0$, is not differentiable at $0$.

\bigskip

To proceed, we require a result from \cite{LW}.

\renewcommand{\theenumi}{\alph{enumi}}

\begin{prop}\cite[Proposition 8]{LW}\label{prop30}
Let $\mathfrak{X}$ be a Banach space and suppose that $\Phi: U\times
\mathfrak{X} \to F$ has the following properties.
\begin{enumerate}
\item For each $u \in \mathfrak{X}$, $\Phi(\cdot,u): U \to F$ belongs
to $C^q(U,F)$.  Denote the $j$-th derivative of this function
by $D^j_y\Phi(\cdot,u)$.
\item For each $y\in U$ and each $j \in \N_q$, $D^j_y\Phi(y,\cdot):\mathfrak{X}\to \cS^{j}(H,F)$ is a bounded linear operator.
\item For each $y_0 \in U$ and each $j \in \N_q$, there exists $\ep
> 0$ so that
\[ \sup_{y \in B(y_0,\ep)}\sup_{\|u\|\leq 1}\|D^j_y\Phi(y,u)\| <
\infty.\]
\end{enumerate}
If $f: U \to \mathfrak{X}$ is a $C^q$ function on $U$, then $\Psi:
U \to F$ defined by $\Psi(y) = \Phi(y,f(y))$ belongs to
$C^q(U,F)$.  For each $m \in \N_q$ and each $y \in U$,
\begin{equation}\label{eq9}
D^m\Psi(y)s^m = \sum^m_{j=0}\binom{m}{j}D^j_y\Phi(y, D^{m-j}f(y)s^{m-j})s^j,
\quad s \in H.
\end{equation}
\end{prop}

Applying Proposition \ref{prop30} to the operator $T$ with the representation given at the beginning of the section, we obtain

\begin{prop}\label{prop31}
For any $k$ and any $S \in \cS^k(\R^n,E)$, the function $y\mapsto \Phi_k(y,S)$ belongs to $C^q(U,F)$. Given $m \in \N_q$, denote the $m$-th derivative of this function by $D^m_y\Phi_k(\cdot,S)$.
For each $y \in U$, the map $S \mapsto D^m_y\Phi_k(y,S)$ is a bounded linear operator from $\cS^k(\R^n,E)$ to $\cS^m(H,F)$. The set of all such operators is locally bounded on $U$; that is, for each $y_0\in U$, there exists $\ep > 0$ so that
\[ \sup_{y \in B(y_0,\ep)}\sup_{\|S\|\leq 1}\|D^m_y\Phi_k(y,S)\| <
\infty.\]
\end{prop}

\begin{proof}
We proceed by induction on $k$. Assume that the proposition holds up to and including $k-1$. (If $k=0$, the assumption is vacuously true.) If $S \in \cS^k(\R^n,E)$, let $f_S:X \to E$ be the function $f_S(x) = Sx^k$. Then
\begin{equation}\label{eq10}
Tf_S(y) = \Phi_k(y,S) + \sum^{k-1}_{i=0}\frac{k!}{(k-i)!}\Phi_i(y,Sh(y)^{k-i}).
\end{equation}
Since $h$ is $C^q$ on $U$ by Proposition \ref{prop29}, it follows from the inductive hypothesis and Proposition \ref{prop30} that the summation term above is a $C^q$ function of $y$ on $U$.
Since $Tf_S \in C^q(Y,F)$ as well, $\Phi_k(y,S)$ is a $C^q$ function of $y$ on $U$.
Fix $y \in U$. By the continuity of $T$, for each $m$ in $\N_q$ and each $y \in Y$, $D^m(Tf_S)(y)$ is a bounded linear operator of the variable $S$.
Furthermore, by the inductive hypothesis and the differentiation formula (\ref{eq9}), the $m$-th derivative with respect to $y$ of the summation term in (\ref{eq10}) is a bounded linear operator of $S$. Thus the map $S\mapsto D^m_y\Phi_k(y,S)$ is a bounded linear operator.
To obtain the final conclusion of the proposition, it suffices to show, in view of (\ref{eq9}), (\ref{eq10}) and the inductive hypothesis, that for any $y_0 \in U$, there exists $\ep > 0$ so that
\[\sup_{y \in B(y_0,\ep)}\sup_{\|S\|\leq 1}\|D^m(Tf_S)(y)\| <
\infty.\]
For each $y\in Y$, we have already observed that the linear operator $R_{m,y} : \cS^k(\R^n,E) \to \cS^m(H,F)$, $R_{m,y}(S) = D^m(Tf_S)(y)$, is bounded. If $(y_j)$ converges to $y_0$, then
\[ R_{m,y_j}(S) = D^m(Tf_S)(y_j) \to D^m(Tf_S)(y_0) = R_{m,y_0}(S)\]
since $Tf_S$ is a $C^q$ function. By the uniform boundedness principle, the sequence of operators $(R_{m,y_j})$ is uniformly bounded. Hence the operators $(R_{m,y})_{y\in Y}$ are locally bounded (in the operator norm) on $Y$, yielding the desired result.
\end{proof}

Now suppose that $T$ is also compact. There exist a compact set $L$ in $X$ and $\ell \in \N_p$ so that $\{Tf: \rho_{L,\ell}(f) <1\}$ is relatively compact in $C^q(Y,F)$. For $k \in \N_p$ and $S \in \cS^k(\R^n,E)$, let $f_S(x) = Sx^k$.
Then $\{Tf_S: \|S\|\leq 1\}$ is relatively compact in $C^q(Y,F)$. Let $K = \{(y_i)\}\cup\{y_0\}$, where $(y_i)$ is a sequence in $Y$ converging to $y_0 \in Y$. For each $m \in \N_q$, $\{D^m(Tf_S)_{|K}: \|S\| \leq 1\}$ is relatively compact in $C(K,\cS^m(H,F))$. Let $R_{m,y}(S) = D^m(Tf_S)(y)$
Thus $\{R_{m,y}(S):\|S\| \leq 1\}$ is relatively compact for each $y \in K$ and the set of maps $y \mapsto R_{m,y}(S)$, $\|S\| \leq 1$, is equicontinuous on $K$. Hence $R_{m,y}$ is a compact operator for any $y \in Y$ (since $K$ can be chosen to contain $y$) and $\lim_{i}R_{m,y_i} = R_{m,y_0}$ in the operator norm of $L(\cS^k(\R^n,E),\cS^m(H,F))$. Since $Y$ is metric, this means that $y\mapsto R_{m,y}$ is continuous with respect to the operator norm topology.

\begin{thm}\label{thm32}
Let $X$ be an open set in $\R^n$ and assume that $0\leq q < p \leq \infty$. Suppose that $T: C^p(X,E) \to C^q(Y,F)$ is a compact, separating, nowhere trivial linear operator with support map $h$.
Then there exist $\ell\in \N_p$ and functions $\Phi_k: Y\times \cS^k(\R^n,E) \to F$, $0 \leq k \leq \ell$, so that $T$ has the representation
\begin{equation*}
Tf(y) = \sum^{\ell}_{k=0}\Phi_k(y, D^kf(h(y)))
\end{equation*}
for all $f\in C^p(X,E)$ and all $y \in Y$.  For a given $y$, $\Phi_k(y,\cdot)$ is a bounded linear operator from $\cS^k(\R^n,E)$ into $F$.  The set $h(Y)$ is relatively compact in $X$. There exists a dense open set $U$ in $Y$ so that $h$ is $C^q$ on $U$.
For each $S \in \cS^k(\R^n,E)$, $\Phi_k(y,S)$ belongs to $C^q(U,F)$ as a function of $y \in U$.
Moreover, for any $k \leq \ell$,
\begin{enumerate}
\item For any $y \in U$ and $m \in \N_q$, the map $D_y^m\Phi_k(y,\cdot): \cS^k(\R^n,E) \to \cS^m(H,F)$ is a compact operator;
\item For each $m \in \N_q$, the map $y\mapsto D_y^m\Phi_k(y,\cdot)$ from $U$ into the space of bounded linear operators $L(\cS^k(\R^n,E), \cS^m(H,F))$ is continuous with respect to the operator norm topology.
\end{enumerate}
\end{thm}

\begin{proof}
By Proposition \ref{prop24}, $T$ has the given representation for some $\ell \in \N_p$ and $h(Y)$ is relatively compact in $X$. By Theorem \ref{thm10}, $\Phi_k(y,\cdot): \cS^k(\R^n,E)\to F$ is a bounded linear operator for each $y \in Y$. By Proposition \ref{prop29}, $h$ is $C^q$ on a dense open subset $U$ of $Y$.  By Proposition \ref{prop31}, for each $S \in \cS^k(\R^n,E)$, $\Phi_k(y,S)$ belongs to $C^q(U,F)$ as a function of $y \in U$. Finally, using the notation above, rewrite equation (\ref{eq10}) as
\[ \Phi_k(y,S) = \sum^{k-1}_{i=0}\frac{k!}{(k-i)!}\Phi_i(y,Sh(y)^{k-i}) - Tf_S(y).\]
For each $m \in \N_q$, observe that $D^m(Tf_S)(y) = R_{m,y}(S)$. By the discussion preceding the theorem, $R_{m,y}$ is a compact operator and a continuous function of $y$ with respect to the operator norm topology.  Differentiate the summation term ($m$-times) using (\ref{eq9}) for $y \in U$. Assuming conditions (a) and (b) for $\Phi_i$, $0\leq i <k$,
the $m$-th derivative with respect to $y$ of the summation term above is a compact operator of $S$, and is a continuous function of $y \in U$ with respect to the operator norm topology.
Therefore, we obtain conditions (a) and (b) for $\Phi_k$.
\end{proof}

\noindent{\bf Remarks}.

\noindent(a) If $\Phi_k(y_0,\cdot) \neq 0$ for some $k \in \N_p$ with $k+q > p$, then $h$ is constant on a neighborhood of $y_0$ by Theorem \ref{thm12}.

\noindent(b) If, in addition to the conditions on $\Phi_k$'s enunciated in the theorem, it is also assumed that for all $m \in \N_q$ and $1\leq k\leq \ell$, the $m$-th derivative of $\sum^{k-1}_{i=0}\frac{k!}{(k-i)!}\Phi_i(y,Sh(y)^{k-i})$ with respect to $y$, considered as a linear operator of the variable $S$, extends continuously in the operator norm topology to a function on $Y$, then the operator $T$ with the given representation is compact.

\bigskip

A similar result holds for weakly compact operators. The argument is similar to that of Theorem \ref{thm32}, but using the consideration for weak compactness as in the proof of Theorem \ref{thm27}. The details are omitted.

\begin{thm}\label{thm33}
Let $X$ be an open set in $\R^n$ and assume that $0\leq q < p \leq \infty$. Suppose that $T: C^p(X,E) \to C^q(Y,F)$ is a weakly compact, separating, nowhere trivial linear operator with support map $h$.
Then there exist $\ell\in \N_p$ and functions $\Phi_k: Y\times \cS^k(\R^n,E) \to F$, $0 \leq k \leq \ell$, so that $T$ has the representation
\begin{equation*}
Tf(y) = \sum^{\ell}_{k=0}\Phi_k(y, D^kf(h(y)))
\end{equation*}
for all $f\in C^p(X,E)$ and all $y \in Y$.  For a given $y$, $\Phi_k(y,\cdot)$ is a bounded linear operator from $\cS^k(\R^n,E)$ into $F$.  The set $h(Y)$ is relatively compact in $X$. There exists a dense open set $U$ in $Y$ so that $h$ is $C^q$ on $U$.
For each $S \in \cS^k(\R^n,E)$, $\Phi_k(y,S)$ belongs to $C^q(U,F)$ as a function of $y \in U$.
Moreover, for any $k \leq \ell$,
\begin{enumerate}
\item For any $y \in U$ and $m \in \N_q$, the map $D_y^m\Phi_k(y,\cdot): \cS^k(\R^n,E) \to \cS^m(H,F)$ is a weakly compact operator;
\item For each $m \in \N_q$, the map $y\mapsto (D_y^m\Phi_k(y,\cdot))^{**}$ from $U$ into the space of bounded linear operators $L(\cS^k(\R^n,E)^{**}, \cS^m(H,F)^{**})$ is continuous with respect to the strong operator topology.
\end{enumerate}
\end{thm}

\end{document}